\documentclass[review,onefignum,onetabnum]{siamonline171218}

\usepackage{amsfonts}
\usepackage{graphicx}
\usepackage{epstopdf}
\usepackage{algorithmic}

\usepackage{bm}
\usepackage{color}
\usepackage{amsmath}
\usepackage{subfigure}
\usepackage{cleveref}
\usepackage{amssymb}
\usepackage[USenglish]{babel}

\graphicspath{ {images/} }

\DeclareMathOperator{\rank}{rank}

\DeclareMathOperator*{\argmin}{\arg\!\min}
\DeclareMathOperator{\Ima}{Im}

\ifpdf
  \DeclareGraphicsExtensions{.eps,.pdf,.png,.jpg}
\else
  \DeclareGraphicsExtensions{.eps}
\fi

\usepackage{enumitem}
\setlist[enumerate]{leftmargin=.5in}
\setlist[itemize]{leftmargin=.5in}

\newsiamremark{remark}{Remark}
\newsiamremark{hypothesis}{Hypothesis}
\crefname{hypothesis}{Hypothesis}{Hypotheses}
\newsiamthm{claim}{Claim}

\headers{Data-driven Discovery of Closure Models}{Shaowu Pan, Karthik Duraisamy}

\title{Data-driven Discovery of Closure Models\thanks{Submitted to the editors 03/25/2018.
\funding{AFOSR grants FA9550-16-1-0309 \& FA9550017-1-0195}}}

\author{
Shaowu Pan\thanks{Department of Aerospace Engineering, University of Michigan, Ann Arbor, MI (\email{shawnpan@umich.edu}, \email{kdur@umich.edu} )}
\and
Karthik Duraisamy\footnotemark[2]
}

\usepackage{amsopn}

\ifpdf   
\hypersetup{
  pdftitle={Data-Driven Discovery of Closure Models},
  pdfauthor={Shaowu Pan \and Karthik Duraisamy}
}
\fi

\begin{document}

\maketitle
\begin{abstract}
Derivation of reduced order representations of dynamical systems requires the modeling of the truncated dynamics on the retained dynamics. In its most general form, this so-called closure model has to account for memory effects. In this work, we present a framework of operator inference to extract the governing dynamics of closure from data in a compact, non-Markovian form. We employ sparse polynomial regression and artificial neural networks to extract the underlying operator. For a special class of non-linear systems, observability of the closure in terms of the resolved dynamics is analyzed and theoretical results are presented on the compactness of the memory. The proposed framework is evaluated on examples consisting of linear to nonlinear systems with and without chaotic dynamics, with an emphasis on predictive performance on unseen data.
\end{abstract}



\section{Introduction}{\label{sec:introduction}}

Complex problems in science and engineering are typically characterized by high-dimensional dynamics. Examples include the modeling of turbulent fluid flows, molecular dynamics, and astrophysical plasmas. When such problems are viewed from a dynamical systems perspective, the high dimensionality of phase space is a consequence of the fact that important physical processes occur over a wide range of spatial and temporal scales. However, effective computational models of these systems for the purposes of analysis, design and control require accurate low-dimensional representations. Popular techniques to obtain low-dimensional representations include projection-based reduced order models~\cite{pod1,pod2,pod3,pod4},  reduced basis methods \cite{Veroy, Rozza}, proper generalized decomposition~\cite{chinesta2011short}, and Krylov subspace techniques~\cite{Beattie}. All of these techniques aim to capture the  dynamics essential to a quantity of interest in by solving for a small number of uknowns (usually by restricting the dynamics to a low-dimensional manifold). In most practical situations, however, the multiscale nature of the problem is such that a low-dimensional representation requires closure. In other words, the influence of the discarded degress of freedom on the retained unknowns becomes important and must be modeled.

The closure problem is well-recognized by the scientific computing community, and is typically addressed by invoking physical and/or mathematical arguments. A pertinent example of physics-based closure is Large Eddy Simulation~\cite{les} (LES) in fluid dynamics, where the impact of the unresolved scales on the resolved scales is often modeled via an eddy diffusivity hypothesis~\cite{germano}. Another example~\cite{kouznetsova2001approach} involves the determination of constitutive properties of complex materials through the detailed modeling of the microstructure. Approximate Green's function-based closures~\cite{vms}, adaptive deconvolution~\cite{stolz2001approximate}, and homogenization techniques~\cite{homo} are representative of mathematically-inspired closures. 

An alternate approach is to pursue data-driven techniques to address closure. There are several instances of the use of data in closure modeling and the following is not intended to be a complete or chronological review, but rather presents  representative references from the viewpoint of the various levels at which data has been used to aid closure modeling. Observational data has been used to calibrate closure parameters in reduced fidelity~\cite{oliver2011bayesian} and reduced order~\cite{couplet2005calibrated,benosman2017learning} models. In these approaches, the functional form of the closure term is prescribed (for instance, via an eddy viscosity assumption) with free parameters which are inferred by minimizing the misfit between the model output and training data. As an example of a more extensive approach,  Xie et al.~\cite{xie2017data} impose a general structure to the closure term and infer matrix operators within the structure. At the next higher level, Ibanez et al.~\cite{ibanez2018manifold} use manifold learning to identify  locally-linear embeddings and construct constitutive relationships for elasticity. Parish et al.~\cite{parish2016paradigm}, Singh et al.~\cite{singh2017machine} directly extract the functional form of augmentations to the closure term by combining statistical inference and learning. 

The goal of this work is to extract closure operators for reduced-dimensional dynamical systems using data snapshots generated from the original high-dimensional dynamical system. The low-dimensional state is augmented with a new set of variables, which represent the closure term, and the evolution equation for the dynamics is discovered in terms of the low-dimensional state using polynomial regression and neural networks. A key difference compared to the literature cited above is that the closed lower-dimensional system is capable of emulating non-Markovian characteristics. Further, the functional form of the evolution equation of the closure is not imposed, but rather extracted directly from the data. 

Over the past few years, much research has been dedicated to the topic of ``data-driven discovery of governing equations,'' using techniques such as dynamic mode decomposition~\cite{dmd}, feature-space regression~\cite{edmd,sindy}, operator inference~\cite{opinf}, and neural networks~\cite{raissi2017physics}, etc. These works have demonstrated that it is possible to a) {\em rediscover} known equations from data, or b) derive approximate representations of systems for which precise equations cannot be written (such as the spread of epidemics~\cite{mangan2016inferring}). The scope of the present work is different, as the structure of the closure is unknown even for simple non-linear dynamical systems. It is, however, assumed that the full-order model corresponding to the high-dimensional system is known (as is the case in many physical problems, such as fluid dynamics where the governing equations are known, but are prohibitively expensive to solve in high-dimensional form) and this knowledge is incorporated into the model formulation process. Furthermore, emphasis is on prediction rather than reconstruction.

This paper is structured as follows: In \cref{sec:problem_description}, the closure problem is briefly described. In \cref{sec:framework}, a framework of operator inference is presented. In \cref{sec:linear} and \cref{sec:non-linear}, applications of this framework with sparse polynomial regression and artificial neural network (ANN) are presented on various problems ranging from simple linear systems to a nonlinear PDE system. Theoretical investigations are conducted on the structure of the closure dynamics in \cref{sec:theory}. Conclusions and perspectives are drawn in \cref{sec:conclusion}.

\section{Description of the closure problem}\label{sec:problem_description}

Consider the autonomous nonlinear dynamical system in \cref{{eq:general_fom}}
\begin{equation}{\label{eq:general_fom}}
\frac{d \bm x}{dt} = \bm F(\bm x),
\end{equation} 
where 
$\bm x(t) \in \mathbb{R}^N $, $N \in \mathbb{N}^+$,
$ t \in [ 0, +\infty ) $, 
$\bm{x}(0) = \bm{x}_0$ and
$\bm{F}(\cdot): \mathbb{R}^{N} \mapsto \mathbb{R}^{N}$.

To serve as a representative lower-dimensional dynamical system, we consider a partition 
\begin{equation}\label{eq:general_fom_partition}
\bm{x} = 
\begin{bmatrix}
\textcolor{black}{\bm{\hat x}} \\ 
\textcolor{black}{\bm{\tilde x}}
\end{bmatrix},  
\bm F(\bm x) = 
\begin{bmatrix}
\bm{\hat{F}}(\textcolor{black}{\bm{\hat x}} , \textcolor{black}{\bm{\tilde x}}) \\ 
\bm{\tilde{F}}(\textcolor{black}{\bm{\hat x}},  \textcolor{black}{\bm{\tilde x}}) 
\end{bmatrix},
\end{equation}
where $\bm{\hat{F}}(\cdot,\cdot): \mathbb{R}^{Q} \times \mathbb{R}^{N-Q} \mapsto \mathbb{R}^{Q}$, $\bm{\tilde{F}}(\cdot,\cdot): \mathbb{R}^{Q} \times \mathbb{R}^{N-Q} \mapsto \mathbb{R}^{N-Q}$, $Q \in \mathbb{N}^{+}$. $\textcolor{black}{\bm{\hat x}} \in \mathbb{R}^Q$ is the low-dimensional or resolved state and $\textcolor{black}{\bm{\tilde x}} \in \mathbb{R}^{N-Q}$ represents the unresolved modes. In general terms, the above partition appears arbitrary. This partition is, however, directly relevant in a number of problems: (i) in projection-based Reduced Order Models (ROMs), where the components of the state in the original dynamical system are ordered according to an energy metric; (ii) in large eddy simulations (LES) of turbulence using spectral or finite element methods, where there is a clear separation between resolved and unresolved scales; (iii) in system identification, where the system is only partially observed and a governing equation for the partially observed system is desired. 

The evolution of the reduced state is given by
\begin{equation}{\label{eq:non-useful}}
\frac{d \bm{\hat{x}}}{dt} =  \bm{\hat{F}}(\textcolor{black}{\bm{\hat x}},\textcolor{black}{\bm{\tilde x}}),
\end{equation} where $\bm{\hat x}(0) = \bm{\hat x}_0 \in \mathbb{R}^{Q}$ and $\bm{\tilde x}(0) = \bm{\tilde x}_0 \in \mathbb{R}^{N-Q}$.

Note that \cref{eq:non-useful} is not very useful, as the trajectory of the unresolved state $\bm{\tilde{x}}(t)$ is present in these equations. In reduced order modeling, a closed set of equations of the form \begin{equation}{\label{eq:ROM}}
\frac{d \textcolor{black}{\bm{\hat{x}}}}{dt} = \bm{F}_{ROM}( \textcolor{black}{\bm{\hat{x}}}),
\end{equation} is obtained through physically-insipired~\cite{couplet2005calibrated,benosman2017learning}, data-augmented~\cite{xu2017reduced} or purely data-driven~\cite{opinf} methods.

In~\ref{eq:ROM}, $\bm{\hat x}(0) = \bm{\hat x}_0$ and $\bm{F}_{ROM}: \mathbb{R}^{Q} \mapsto \mathbb{R}^{Q}$.
When $\bm{F}_{ROM}(\textcolor{black}{\bm{\hat{x}}}) = \bm{\hat{F}}(\textcolor{black}{\bm{\hat x}} , \bm{0})$, one obtains a classic truncated ROM. In fluid dynamic modeling terms, this corresponds to a Large Eddy Simulation without a explicit subgrid scale model. In obtaining such approximations, a key fact to consider is that, even if the high-dimensional system is Markovian, the corresponding projected low-dimensional system can be non-Markovian. This is evident even in the simplest case of a linear system. Consider that the full order model with its partition into resolved and unresolved states:

\begin{equation}{\label{eq:general_fom_linear}}
\dfrac{d}{dt}
\begin{bmatrix}
\textcolor{black}{\bm{\hat x}}        \\
\textcolor{black}{\bm{\tilde x}}
\end{bmatrix}
=
\begin{bmatrix}
A_{11} & A_{12}\\
A_{21} & A_{22}
\end{bmatrix}
\begin{bmatrix}
\textcolor{black}{\bm{\hat x}}         \\
\textcolor{black}{\bm{\tilde x}}
\end{bmatrix}.
\end{equation}

The evolution of the resolved states is given by:

\begin{equation}{\label{eq:linear_gle}}
\frac{d \textcolor{black}{\bm{\hat x}}}{dt}
=
A_{11}\bm{\hat x} +
\textcolor{black}{\int_{0}^{t} A_{12} e^{A_{22}\tau } A_{21} \bm{\hat x}(t-\tau) d\tau} + A_{12} e^{A_{22}t}\textcolor{black}{\bm{\tilde{x}}}(0).
\end{equation}

This equation is closed in the resolved variables. The first term on the right hand side is $\bm{\hat{F}}(\textcolor{black}{\bm{\hat x}},\bm{0})$; the second term, which represents the closure,  involves the time-history of the resolved modes, and the third term involves the initial condition of the unresolved state, and can be expected to decay in time for a dissipative system. 

The above expression of closure can be generalized to any nonlinear system~\cref{eq:general_fom,eq:general_fom_partition} using the Mori-Zwanzig formalism~\cite{chorin2002optimal,chorin2014estimating},  and exact evolution equations can be written for the reduced state in the following Generalized Langevin form:
\begin{equation}{\label{eq:general_MZ}}
\dfrac{d \textcolor{black}{\bm{\hat x}}}{dt} = \bm{\hat{F}}(\textcolor{black}{\bm{\hat x}},\bm{0}) + \ \int_0^t \mathcal{K}(\bm{\hat{x}}(s), t - s) \ ds \ +  \mathcal{Q}(\bm{\hat{x}}(t)),
\end{equation} 
where $\mathcal{K}$, $\mathcal{Q}$ are complex operators associated with convolution and influence of the initial conditions, respectively.


While the above equation is mathematically precise and formally closed in the resolved variables, the functional forms of $\mathcal{K}$ and $\mathcal{Q}$ are not explicitly known and numerically intractable, even for the simplest non-linear dynamical systems, and must thus be determined via the solution of another high-dimensional partial differential equation~\cite{gouasmi2017priori}. The key message is that reduced-order representations of even a linear Markovian system can introduce memory or time-history effects in an explicit form that requires \emph{all} its previous states. In the present work, we introduce a dynamic memory and aim to extract its evolution using operator inference. It is also shown that for a specialized
class of nonlinear systems that the memory length is compact, and thus the full history of resolved states is not necessary.

\section{Framework of operator inference}{\label{sec:framework}}

As indicated by the Mori-Zwanzig formalism, the exact closure is a compositional convolution operator on all past resolved states. This approach is equivalent to the concept of dynamic or recurrent memory~\cite{goodfellow2016deep}, a concept which has been very attractive in the time series modeling and deep learning communities. To address complex memory structures, we consider time delay vectors in the  framework as implied by Takens embedding theorem~\cite{takens1981detecting} which states that there exists a diffeomorphism between proper time delayed (reconstructed) attractor and the original manifold. 

By leveraging both dynamic memory and implications of Takens embedding theorem, a framework of operator inference is proposed as shown in the following discretized augmented system:

\begin{align}{\label{eq:general_model_1}}
\frac{D \bm{\hat{x}}}{Dt} &= \bm{\hat{F}}(\bm{\hat{x}}(t),\bm{0}) + \textcolor{black}{\bm{\delta}(t)}, \\
{\label{eq:general_model_2}}
\frac{D \textcolor{black}{\bm{\delta}}}{Dt} &=  \bm{G}(\bm{\hat{x}}(t - s_0), 
\ldots,
\bm{\hat{x}}(t-s_p),
\textcolor{black}{\bm{\delta}(t - s_0)},,
\ldots,
\textcolor{black}{\bm{\delta}(t-s_p)}),
\end{align} where $\bm{\delta} \in \mathbb{R}^{Q}$ is the closure term, $p \in \mathbb{N}$ is the number of delays of past time information,  and $\frac{D}{Dt}$ represents time discretization. $\{s_i\}_{i=0}^{p}$ is given as a strictly monotonic equally spaced time sequence with $s_i = i\Delta t$, $s_i \in [0,t]$.

We note that Shulkind et al.~\cite{shulkind2017experimental} also pursue closure, but are focused on 
developing a Markovian correction term with the restriction that the magnitude of the closure term is small compared to the resolved term. In the present work,
memory effects are represented via an additional governing equation for $\bm \delta$. Two important features of the current framework are: a) the functional form of $\bm{{G}}$ is extracted from data, i.e., solution snapshots from the high-dimensional model, and b) this framework is inherently non-Markovian (for the resolved variables $\bm{\hat{x}})$. As a side note in~\cref{app:compare_elman}, the current framework of operator inference for the shortest memory case $p=0$ is compared to Elman's network~\cite{elman1990finding}, a forerunner to modern recurrent neural networks~\cite{goodfellow2016deep}, to highlight similarities and differences.

\subsection{Interpretation of operator inference}

In a simple setting, assume the dynamical system above is discretized using first-order forward time integration.  Rewrite the partitioned system as

\begin{align}
\frac{\bm{\hat{x}}^{n+1}-\bm{\hat{x}}^{n}}{\Delta t} &= \bm{\hat{F}}(\bm{\hat{x}}^{n},\bm{0}) + \bm{\hat{F}}(\bm{\hat{x}}^{n},\bm{\tilde{x}}^{n}) - \bm{\hat{F}}(\bm{\hat{x}}^{n},\bm{0}) = \bm{\hat{F}}(\bm{\hat{x}}^{n},\bm{0}) + \bm{\delta}^{n}, \\
\frac{\bm{\tilde{x}}^{n+1}-\bm{\tilde{x}}^{n}}{\Delta t} &= \bm{\tilde{F}}(\bm{\hat{x}}^{n}, \bm{\tilde{x}}^{n}),
\end{align} 
where 
$\bm{\delta} = \bm{\hat{F}}(\bm{\hat{x}}, \bm{\tilde{x}}) - \bm{\hat{F}}(\bm{\hat{x}}, \bm{0}) \triangleq \bm{R}(\bm{\hat{x}}, \bm{\tilde{x}})$. Note that $\bm{\delta}^{n+1} = \bm{\hat{F}}(\bm{\hat{x}}^{n+1}, \bm{\tilde{x}}^{n+1}) - \bm{\hat{F}}(\bm{\hat{x}}^{n+1}, \bm{0}) = \bm{R}(\bm{\hat{x}}^{n+1}, \bm{\tilde{x}}^{n+1}) = \bm{R}(\bm{\hat{x}}^{n} + \Delta t \bm{\hat{F}}(\bm{\hat{x}}^{n}, \bm{0}) + \Delta t \bm{\delta}^{n}, \bm{\tilde{x}}^{n} + \Delta t \bm{\tilde{F}}(\bm{\hat{x}}^{n}, \bm{\tilde{x}}^{n}))$. Thus, one must obtain $\bm{\tilde{x}}^{n}$ to further evolve the closure.

As implied by the Takens embedding theorem, it is possible to use the information of past resolved states to obtain $\bm{\tilde{x}}^{n}$. Considering a time delay up to $p$ steps, the equations that involve $\bm{\tilde{x}}$ are given as follows:
\begin{align}
\bm{\delta}^{n} &= \bm{R}(\bm{\hat{x}}^{n}, \bm{\tilde{x}}^{n}),\\
\frac{\bm{\tilde{x}}^{n}-\bm{\tilde{x}}^{n-1}}{\Delta t} &= \bm{\tilde{F}}(\bm{\hat{x}}^{n-1},\bm{\tilde{x}}^{n-1}), \\
\bm{\delta}^{n-1} &= \bm{R}(\bm{\hat{x}}^{n-1}, \bm{\tilde{x}}^{n-1}), \\
\vdots \nonumber
\\
\frac{\bm{\tilde{x}}^{n-p+1}-\bm{\tilde{x}}^{n-p}}{\Delta t} &= \bm{\tilde{F}}(\bm{\hat{x}}^{n-p},\bm{\tilde{x}}^{n-p}),\\
\bm{\delta}^{n-p} &= \bm{R}(\bm{\hat{x}}^{n-p}, \bm{\tilde{x}}^{n-p}) ,
\end{align} 
with the number of equations, component-wise, being $N_{eq} = pN+Q$ and the number of unknowns, component-wise, being $N_{unk} = (p+1)(N-Q)$. Note that $N_{eq} - N_{unk} = (p+2)Q - N$. Therefore, for large enough $p$, it should be possible to determine $\bm{\tilde{x}}^{n}$ from $\bm{\hat{x}}^{n-1}, \ldots, \bm{\hat{x}}^{n-p}$ and $\bm{\delta}^{n}, \ldots, \bm{\delta}^{n-p}$ by solving the nonlinear algebraic equations above. Once $\bm{\tilde{x}}^{n}$ is determined, one can obtain $\frac{\bm{\delta}^{n+1} - \bm{\delta}^{n}}{\Delta t}$ as a function $\bm{G}$ of $\bm{\hat{x}}^{n}, \ldots, \bm{\hat{x}}^{n-p}$ and $\bm{\delta}^{n}, \ldots, \bm{\delta}^{n-p}$. This suggests the possibility of finding $\bm{G}$ through a data-driven method.

\subsection{Definition and data preparation}

As discussed above, the goal of the operator inference framework is to determine $\bm{{G}}$ in \cref{eq:general_model_2}. This process can also be viewed as a nonlinear system identification problem by considering $\bm{\delta}$ as the states of an undetermined system and $\bm{\hat{x}}$ as inputs to this system. Our approach requires the parameterization of $\bm{G}(\cdot)$ in the form of $\bm{G}_{\bm{W}}$ using two different methodologies and then solving an optimization problem over the parameter space ${\bm{W}}$. The first parametrization method is sparse polynomial regression (similar to the SINDy approach of Brunton et al.~\cite{sindy}) which leverages the fact that many dynamical systems in science and engineering can be represented as a sparse combination of monomials. The second method uses time delayed neural networks~\cite{waibel1990phoneme} which are scalable to high-dimensional nonlinear systems and possesses the universal approximator property and implicit feature selection. Note that for simplicity, the time derivative is realized by a first-order forward scheme throughout this work.

Assume $M$ temporally sequential snapshots $\bm{X} = \begin{bmatrix}\bm{\hat{x}}^j \end{bmatrix}_{j \in I} \in \mathbb{R}^{M \times Q}$ spaced uniformly with a time interval $\Delta t$, i.e., $s_i = i\Delta t$, $\forall i \in \{0, \ldots, p\}$ and $\bm{dX} = \begin{bmatrix}D\bm{\hat{x}}^j/Dt\end{bmatrix}_{j \in I} \in \mathbb{R}^{M \times Q}$ obtained from the full order model (Eq.\cref{{eq:general_fom}}).  Here, $I = \{j|j\in \mathbb{N}^{+}, 1\le j \le M\}$, $M \in \mathbb{N}^{+}$ and $p \in \mathbb{N}$ is the number of steps of past memory. We divide $\bm{X}$ into training and testing data through the index set, considering $p$ time delayes: $I^{p} = \{ j | j \in I, \forall i \in \mathbb{N}, 0 \le i \le p, j-i \in I \}$; training data index set: $I^{p}_{train} = \{j |   j\in I^p ,j\le M_{train} \} $, $M_{train} \in \mathbb{N}^{+}$; testing data index set: $I^{p}_{test} = I^{p} \setminus I^{p}_{train}$. It should be noted that we simply choose $s_i = i\Delta t$, where $\Delta t$ is the time interval between equally spaced snapshots. As a result, given $\bm{\hat{F}}(\cdot,0)$, the corresponding snapshots of training closure are
\begin{equation}{\label{eq:data_closure}}
\bm{\Delta}_{I^{p}_{train}}= \begin{bmatrix}
\bm{\delta}^j
\end{bmatrix}_{j \in I^{p}_{train}} = 
\begin{bmatrix}
D\bm{\hat{x}}^j/Dt - \bm{\hat{F}}(\bm{\hat{x}}^j,\bm{0})
\end{bmatrix}_{j \in I^{p}_{train}}.
\end{equation}

Therefore, the time-delayed feature matrix of $\bm{\hat{x}}$ and $\bm{\delta}$ in the training data can be constructed as
\begin{equation}{\label{eq:feature_train}}
\bm{Y}_{I^{p}_{train}} = \begin{bmatrix}
\bm{\hat{x}}^j, \ldots ,\bm{\hat{x}}^{j-p}, \bm{\delta}^j,\ldots, \bm{\delta}^{j-p}
\end{bmatrix}_{j \in I^p_{train}}  = \begin{bmatrix} \bm{y}^{j}  \end{bmatrix}_{j \in I^p_{train}},
\end{equation}
where $\bm{y}^{j} \in \mathbb{R}^{2(1+p)Q}$.

Considering the dependency indicated by the relation between the sequences of $\bm{\delta}$ and $\bm{\hat{x}}$, $\forall j \in \mathbb{N}^{+}$, $j \le n$, $\bm{\delta}^{n-j} = \frac{\bm{\hat{x}}^{n-j+1}-\bm{\hat{x}}^{n-j}}{\Delta t} - \bm{\hat{F}}({\bm{\hat{x}}^{n-j},\bm{0}})$; the economic time-delayed feature matrix can be constructed as
\begin{equation}{\label{eq:feature_train_eco}}
\bm{Y}^{eco}_{I^{p}_{train}} = \begin{bmatrix}
\bm{\hat{x}}^j, \ldots ,\bm{\hat{x}}^{j-p}, \bm{\delta}^j
\end{bmatrix}_{j \in I^p_{train}}  = \begin{bmatrix} \bm{y}^{j}_{eco}  \end{bmatrix}_{j \in I^p_{train}},
\end{equation}
where $\bm{y}^{j}_{eco} \in \mathbb{R}^{(2+p)Q}$.

The training target is
\begin{equation}{\label{eq:snapshots_dddt_train}}
\bm{Z}_{I^{p}_{train}} = {D\bm{\Delta}_{I^{p}_{train}}}/{Dt}  = 
[D\bm{\delta}^j/Dt ]_{j \in I^p_{train}} = [\bm{z}^{j} ]_{j \in I^p_{train}},
\end{equation}
where $\bm{z}^{j} \in \mathbb{R}^{Q}$.

Likewise, the testing feature matrix and target are
\begin{equation}{\label{eq:feature_test}}
\bm{Y}_{I^p_{test}} = \begin{bmatrix}
\bm{\hat{x}}^j, \ldots ,\bm{\hat{x}}^{j-p}, \bm{\delta}^j,\ldots, \bm{\delta}^{j-p}
\end{bmatrix}_{j \in I^p_{test}} = \begin{bmatrix} \bm{y}^{j}  \end{bmatrix}_{j \in I^p_{test}},
\end{equation} 
\begin{equation}{\label{eq:snapshots_dddt_test}}
\bm{Z}_{I^p_{test}} = \frac{d\bm{\Delta}}{dt}   = [\bm{z}^{j} ]_{j \in I^p_{test}},
\end{equation}
and the corresponding economic feature matrix is
\begin{equation}{\label{eq:feature_test_eco}}
\bm{Y}^{eco}_{I^p_{test}} = \begin{bmatrix}
\bm{\hat{x}}^j, \ldots ,\bm{\hat{x}}^{j-p}, \bm{\delta}^j
\end{bmatrix}_{j \in I^p_{test}} = \begin{bmatrix} \bm{y}^{j}_{eco}  \end{bmatrix}_{j \in I^p_{test}}.
\end{equation}

\subsection{Data driven modeling}

Based on the above definitions, the general idea is to transform the functional approximation problem into an optimization problem either through sparse polynomial regression (SINDy) or neural networks as described in the following subsections.


\subsubsection{Sparse polynomial model}

Since polynomial features frequently appear in many science and engineering applications, polynomial regression~\cite{sindy} is typically a popular choice. An approximation of the form $\bm{G} = \bm{\hat{G}}_{\bm{W}} = \bm{\Theta}_{k} \bm{W}$ is employed by transforming the problem of finding a sparse representation of dynamical system into a convex optimization problem in \cref{eq:lasso}, where $\bm{W} \in \mathbb{R}^{L^p_k \times Q}$ is a matrix of decision variables, and $L^p_k \in \mathbb{N}^{+}$ is the number of component-wise polynomial features up to total degree $k$ of resolved states and using previous $p$ steps. To illustrate this idea, given row vector $\bm{h} \in \mathbb{R}^{1 \times n}$, $\forall n\in \mathbb{N}^{+}$, $\bm{\Theta}_{k}$  is a corresponding feature row vector from a monomial library with a certain maximum total order $k$
\begin{equation}{\label{eq:theta_feature}}
\bm{\Theta}_k(\bm{h}) = \begin{bmatrix}
1 \quad \bm{h} \quad \bm{h}^{P_2} \cdots \bm{h}^{P_k} 
\end{bmatrix}, 
\end{equation}
where $\bm{h}^{P_k}$ refers to all product terms of monomials with total degree $k$. Naturally, $\forall m, n \in \mathbb{N}^{+}$, $\bm{H} \in \mathbb{R}^{m \times n}$, $\bm{\Theta}_k(\bm{H})$ is the row stack of $\bm{\Theta}_k(\bm{H}_i)$, $ i \in \{ 1,\ldots,m \}$, $\bm{H} = [\bm{H}^{T}_1,\ldots,\bm{H}^{T}_m]^{T}$. 

The basic idea of sparse polynomial regression is to find a sparse representation by sparsifying the coefficient matrix $\bm{W}$ from a predefined feature library through either a sequential thresholded least-squares algorithm~\cite{sindy} or using \emph{lasso}~\cite{mangan2016inferring}. 

Applying \emph{lasso}~\cite{tibshirani1996regression,tibshirani2013lasso}, given $|I^{p}_{train}| \ge (1+p)L^p_k$, \cref{eq:lasso} can be shown to be a convex optimization problem for $\bm{w}_{k}$, the $k$th column of $\bm{W}$:
\begin{equation}{\label{eq:lasso}}
\bm{w}_{k}^{*} = \argmin_{\bm{w}_k}\frac{1}{\left\lvert I^p_{train} \right\rvert} \sum_{j \in I^p_{train}}{ \left\lVert \bm{z}^j - \bm{\Theta}_{k}(\bm{y}^{j}) \bm{w}_k \right\rVert_2^2} + \lambda \left\lVert \bm{w}_k \right\rVert_1,
\end{equation} 
where $\vert \cdot \vert$ is the cardinality, $\bm{\Theta}_k(\bm{y}^j) \in \mathbb{R}^{|I^p_{train}| \times L^p_k}$, and $\lambda \in \mathbb{R}^{+}$ is the penalty coefficient. It is important to note that for \emph{lasso} to achieve ideal support recovery, the following constraint must be satisfied \cite{su2017false}
\begin{equation}
n_k/n_p \le \delta/(2 \log n_p) (1+o(1)),
\end{equation}
where $n_k$ is the number of true features, $n_p$ is the number of total features, $\delta = n/n_p$ where $n$ is the number of i.i.d data points. The simplest way to achieve this is to ensure one has a large number data $n$ compared to number of total features $n_p$. \emph{lasso} is implemented via Scikit-learn packages \cite{scikit-learn}.

Regarding the dependency between features, since $j \le n$, $\bm{\delta}^{n-j} = \frac{\bm{\hat{x}}^{n-j+1}-\bm{\hat{x}}^{n-j}}{\Delta t} - \bm{\hat{F}}({\bm{\hat{x}}^{n-j},\bm{0}})$. Hence, if $\bm{\hat{F}}$ is compact in monomials, it is possible to replace $\bm{\delta}^{n-1},\ldots,\bm{\delta}^{n-p}$ by polynomials of $\bm{\hat{x}}^{n},\ldots, \bm{\hat{x}}^{n-p}$. However, if $\bm{\hat{F}}$ is not in polynomial form or if $\bm{\hat{F}}$ contains a very high order polynomial of $\bm{\hat{x}}$, the size of the library will be extremely large and potentially non-convergent. While reuse of $\bm{\delta}^{n-j}$ can alleviate this issue, the trade-off involves using twice the degrees of freedom. For the sparse polynomial regression model, employment of $\bm{\delta}^{n-j}$ is considered throughout this work for better fitting performance.



The current framework with polynomial features is different from the operator inference of Peherstorfer and Willcox~\cite{opinf}, which targets  the entire system and can be viewed as a particular case of the present work with $p=0$, $k=2$, and $\bm{\hat{F}} = 0$, and with POD as preprocessing for dimension reduction. Additionally, sparsity is not encouraged and no memory effects are required in their model. The present framework thus seeks a more general non-Markovian operator inference.

\subsubsection{Artificial neural network model}

In the previous subsection, the number of polynomial features in the feature library is found to grow exponentially with $Q$. One of the most popular tools for efficient high-dimensional functional approximations is the artificial neural network (ANN). The appealing feature of neural networks with a single hidden layer and squashing nonlinear activation function is that any Borel measurable function can be approximated to any degree of accuracy on a compact domain. This property is guaranteed by the universal approximation theory~\cite{Hornik1989}. 
ANN has attracted considerable attention in recent years. The success of deep learning (na{\"i}vely and narrowly speaking for supervised learning, ANN with a large number of hidden layers) lies in learning low-dimensional representations from high-dimensional, complex data effectively and building relationships between learned features and the target~\cite{goodfellow2016deep}.


To parametrize the model described in \cref{{eq:general_model_time_parallel}}, the standard feedforward neural network structure shown in \cref{fig:fnn_plain} is employed for $\bm{G} = \bm{\hat{G}}_{\bm{\theta}, \bm{b}}$. Due to the previously mentioned dependency between sequences of $\bm{\delta}$ and $\bm{\hat{x}}$ and the universal approximator property of ANN, $\bm{y}_{eco} \in \mathbb{R}^{(2+p)Q}$ is used as input. To construct a densely connected feedforward neural network $\bm{\hat{G}}_{\bm{\theta}, \bm{b}}$: $\mathbb{R}^{(2+p)Q} \mapsto \mathbb{R}^Q$ with $L-1$ hidden layers and a linear output layer, the following recursive expression is used for each hidden layer:
\begin{equation}{\label{eq:recursive}}
\bm \eta ^l = \sigma_l (\theta_l \bm \eta^{l-1} + b_l),
\end{equation}
for $l=1,\dots,L-1$, where $\bm{\eta}^{0}$ stands for the input of the neural network, $\bm \eta^l \in \mathbb{R}^{n_l \times 1}$, $n_l \in \mathbb{N^{+}}$ is the number of units in layer $l$, $\theta_l \in \mathbb{R}^{n_{l} \times  n_{l-1}}$, $b_l \in \mathbb{R}^{n_l \times 1}$, $n_0 = (2+p)Q$, and $\sigma_l$ is the activation function of layer $l$. 
Note that the output layer is linear, i.e., $\sigma_L(x)=x$:
\begin{equation}
 \bm{\hat{G}}(\bm{y}_{eco}; \bm{\theta}, \bm b) = \bm \eta^L = \theta_L \bm \eta^{L-1} + b_L,
 \end{equation}
where $\theta_L \in \mathbb{R}^{n_{L} \times n_{L-1}}$, $b_l \in \mathbb{R}^{ n_L \times 1}$, and $n_L = Q$. Parameters of the neural network are summarized as $\bm{W} = \{ \bm{\theta}, \bm{b}  \}$: weights $\bm{\theta} = \{ \theta_j \}_{j=1,\dots,L}$ and biases $\bm b = \{ b_j \}_{j=1,\dots,L}$.
In this work we use two hidden layers where $L=3$ and hidden units are all the same. The full expression of the neural network model is
\begin{equation}{\label{eq:full_basic_model}}
\bm{\hat{G}}(\bm{y}_{eco};\bm W)= \bm{\hat{G}}(\bm{y}_{eco};\bm{\theta}, \bm{b}) = \theta_3\sigma(\theta_2 \sigma(\theta_1 \bm{y}_{eco} + b_1 ) + b_2) + b_3,
\end{equation}
where $\sigma(\cdot)$: $\mathbb{R} \mapsto \mathbb{R}$ is a nonlinear activation function, e.g., ReLU, SeLU, tanh \cite{tensorflow}. 
\begin{figure}[H]
\centering 
\includegraphics[width=0.45\textwidth]{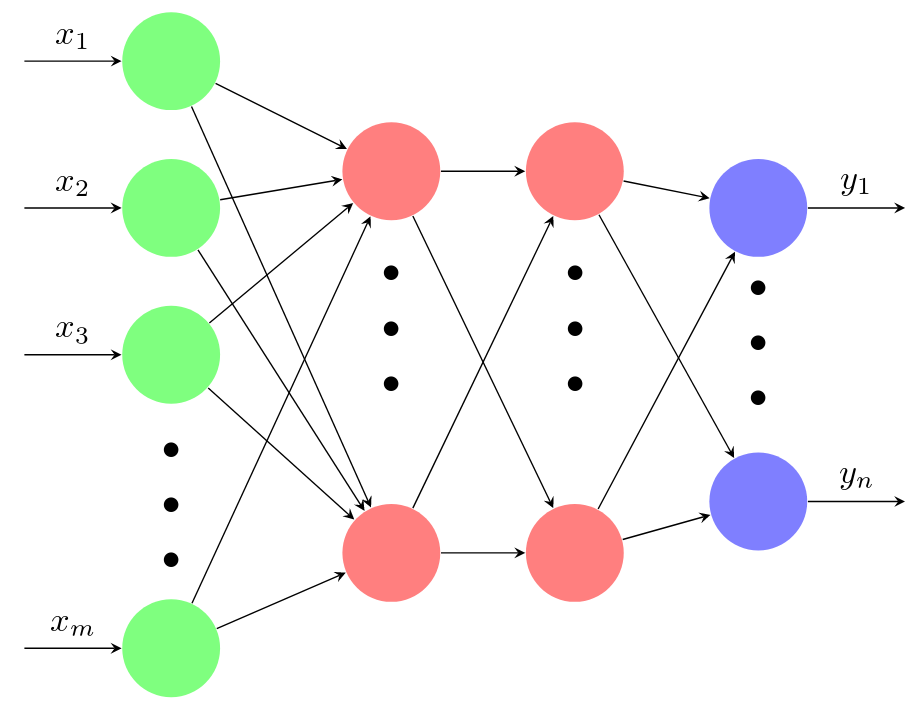}
\label{fig:fnn_plain}
\caption{Schema of a typical feedforward neural network $\mathbb{R}^m \mapsto \mathbb{R}^n$ with two hidden layers with $\bm{x}$ as input and $\bm{y}$ as output}
\end{figure}

The problem of finding a good neural network model is equivalent to searching for a set of parameters $\bm{\theta}$ and $\bm{b}$ that optimize the mean-square-error on training data with weight decay regularization
\begin{equation}{\label{eq:fnn}}
\bm{W}^{*} = \{\bm{\theta}^{*}, \bm{b}^{*}\} = \argmin_{\bm{\theta}, \bm{b}}\frac{1}{\left\lvert I^p_{train} \right\rvert} \sum_{j \in I^p_{train}}{ \left\lVert \bm{z}^j - \bm{\hat{G}}(\bm{y}_{eco}^{j}; \bm{\theta}, \bm b)  \right\rVert_2^2} + \lambda \sum_{l=1}^{L} \left\lVert {\theta}_l \right\rVert^2_F,
\end{equation}
where $\lVert(\cdot)\rVert_F$ denotes the Frobenius norm. The weights are initialized using the standard truncated normal distribution, and a first order gradient-based technique~\cite{kingma2014adam} is used for optimization. Unfortunately, due to the non-convex nature of \cref{eq:full_basic_model}, one can often only afford to find a local minimum instead of the global minimum. However, in practice, a local minimum is usually satisfactory if it is properly trained and validated. Hyperparameters for each case given below are selected using grid search in a certain range. The model is implemented with the Keras \cite{chollet2015keras} and Tensorflow libraries \cite{tensorflow}.




\subsection{Reducing the computational complexity of multi-time effects}

From a training perspective, the most difficult part of generating the polynomial model is in characterizing multi-time effects. The most general way of treating the memory effect in \cref{eq:general_model_1} is to consider all interactions between the past, i.e., cross-time memory effects, similar to a nonlinear autoregressive network with exogenous inputs (NARX) model \cite{billings2013nonlinear} with multi-time correlations. Unfortunately, this method of discovering ${\bm G}$ is severely hindered by the curse of dimensionality. Introducing full interactions with polynomial features up to a total degree $k \in \mathbb{N}^{+}$ would lead to a number of features scaling as $L^p_k = {{2(1+p)Q+k}\choose {k}} \propto (2(1+p)Q)^k$. 

An alternative strategy to build computationally feasible set of features is to \emph{assume} that full memory effects can be approximated with linear superposition of multi-time nonlinear features, i.e., linear treatment of multi-time memory in the form
\begin{equation}{\label{eq:general_model_time_parallel}}
\bm{G}(\bm{\hat{x}}(t - s_0), 
\ldots,
\bm{\hat{x}}(t-s_p),
\textcolor{black}{\bm{\delta}(t - s_0)},,
\ldots,
\textcolor{black}{\bm{\delta}(t-s_p)};\bm{W}) \approx  
\sum_{i=0}^{p}\bm{{G}}_i(\bm{\hat{x}}(t-s_i), {\bm{\delta}(t-s_i)} ; \bm{W}^{i}),
\end{equation} 
where $\bm{{G}}_i(\cdot,\cdot):\mathbb{R}^{Q} \times \mathbb{R}^{Q} \mapsto \mathbb{R}^{Q}$ represents the contribution of the system at $t=t-s_i$ to $t$ in closure dynamics. 
With a polynomial basis, this approach will reduce the number of features up to $k$ total degrees to $L^{p}_{k,reduced} = (1+p){{2Q+k\choose{k}}} - p \propto p(2Q)^k$, which grows linearly with $p$ \footnote{$-p$ comes from removing the redundant constant feature}.

The above assumption would lead to a reduction in the number of fitting parameters with respect to increasing memory length $p$. We note that, if this decoupling is applied to the time delay neural network (TDNN) model, this strategy can be viewed as a regularization of the neural network model by pruning weights between units of different time instances.

\subsection{Model selection}

Since both types of models mentioned above require the specification of hyperparameters before training, model selection is an important aspect. 
For the sparse polynomial regression model, the associated hyperparameters are:
\begin{itemize}
\item maximum total degree of polynomials: $k$
\item maximum number of previous states: $p$
\item penalty coefficient: $\lambda$
\end{itemize}

In practice, we choose $p$ and $k$ heuristically and as small as possible, while still fitting the model with ordinary-least-square (OLS) and keeping the total number of features smaller than the number of samples to ensure strict convexity of the \emph{lasso} problem \cite{tibshirani2013lasso}. 
To determine $\lambda$, we draw the \emph{lasso} path to decide the most appropriate solution that balances complexity and mean-squared-error (MSE) error. It should be noted that when drawing the \emph{lasso} path, we split the training data in time; the first 80\% is training data used to compute the \emph{lasso} path, and the last 20\% is validation data. The goal is to obtain a robust model that generalizes beyond the training set.

For the neural network, a logical strategy of hyperparameter selection has proved to be challenging for even the simplest standard feedforward neural network. In the present work, since the problem size is small, we choose hyperparameters via simple grid search for the type of activation function and number of hidden units. 

\subsection{Evaluation of MSE a priori and a posteriori}
Notice that the optimization problems described in \cref{eq:lasso,eq:fnn} only guarantee performance in an a priori sense on training data. A proper evaluation of the model should be performed both in an a priori sense as mean-squared-error over the data index set $I^p$, 
\begin{equation}{\label{eq:min_general_apriori}}
e^{apr}_{\rm{MSE}} = \frac{1}{\left\lvert I^p \right\rvert} \sum_{j \in I^p}{ \left\lVert \bm{z}^j - {\bm{\hat{G}}}(\bm{y}^j;{\bm{W}^*)} \right\rVert_2^2}, 
\end{equation} 
and in an a posteriori sense in which only the initial condition is given to the model, as
\begin{equation}
e^{apo}_{\rm{MSE}} = \frac{1}{\left\lvert I^p \right\rvert}\sum_{j \in I^p}{ \left\lVert \bm{\hat{x}}^j - \bm{\hat{x}}^{*j}  \right\rVert_2^2} ,
\end{equation}
where $\bm{\hat{x}}^{*}(t)$ is the solution of the augmented system \cref{eq:general_model_1,eq:general_model_2} with $\bm{\hat{G}}(\cdot;\bm{W}^*)$ starting with an exact initial condition. This type of a posteriori evaluation is also called free-run in the time series modeling community~\cite{aguirre2009modeling}.

\section{Results - Linear system} \label{sec:linear}
To illustrate the idea of applying operator inference and to motivate further developments, the polynomial closure model is first applied on a three dimensional linear system shown below

\begin{equation}{\label{eq:3d_linear_system}}
\dfrac{d}{dt}
\begin{bmatrix}
    x_1\\ x_2\\ x_3
    \end{bmatrix}    = 
    \begin{bmatrix}    0 & -1 & -1\\
    0.5 & -1.1 & 1.5\\
    1 & -3 & 0.5
    \end{bmatrix}
    \begin{bmatrix}
    x_1 \\
    x_2 \\
    x_3
    \end{bmatrix},
\end{equation}
where ${\hat{x}}(0) = {\hat{x}}^{0} = x_{1}^{0} = 1$ and $x_{2}^{0} = x_{3}^{0} = 0$. The first-order forward discretized form of \cref{eq:3d_linear_system} is shown in \cref{eq:3d_linear_system_disc} with total degrees of freedom $N=3$ and number of reduced states $Q=1$
\begin{equation}{\label{eq:3d_linear_system_disc}}
    \begin{bmatrix}
    x^{n+1}_1\\ x^{n+1}_2\\ x^{n+1}_3
    \end{bmatrix}
    = 
    \begin{bmatrix}
    x^{n}_1\\ x^{n}_2\\ x^{n}_3
    \end{bmatrix}
    + \Delta t
    \begin{bmatrix}
    0 & -1 & -1\\
    0.5 & -1.1 & 1.5\\
    1 & -3 & 0.5
    \end{bmatrix}
    \begin{bmatrix}
    x^{n}_1 \\
    x^{n}_2 \\
    x^{n}_3
    \end{bmatrix}.
\end{equation}


Consequently, following the operator inference framework with the polynomial form in \cref{eq:general_model_1,eq:general_model_2} and a linear superposition assumption of multi-time effects, we have ${{\hat{x}}} = x_1$, ${\hat{F}} = 0$ for the following ROM formulation
\begin{align}
\hat{x}^{n+1} &= \hat{x}^{n} + \Delta t \delta^n, \\
\delta^{n+1} &= \delta^{n} + \Delta t \sum_{i=0}^{p}\bm{G}_i.
\end{align}

The goal is to extract a functional form of the governing equation $\{\bm{G}_i\}_{i=0}^{p}$ for closure $\delta$ from data, i.e., to determine $(\delta^{n+1} - \delta^{n})/\Delta t$ as a function of previous $\hat{x}$ and $\delta$ from data. The true closure is ${\delta} = { -x_2-x_3}$, $\bm{\tilde{x}} = [x_2^\top, x_3^\top]^\top$, which is assumed to be unknown to the ROM. The simulation is run for $t \in [0,40]$ and $\Delta t =0.01$, resulting in a collection of 4000 snapshots in $\{ \hat{x}, \delta \}$. The first 10\% of data is set for training and the remaining 90\% as testing data. 

For this 3D linear system, the exact solution for the closure dynamics is 
\begin{align}{\label{eq:3d_linear_analytical}}
\dfrac{\delta^{n+1} - \delta^{n}}{\Delta t} = \left (\dfrac{3}{2} - \dfrac{17\Delta t}{20} \right) \hat{x}^{n-1} &- \left( \dfrac{ 183\Delta t + \left({35\Delta t + 10}\right) \left(\dfrac{1}{\Delta t} - \dfrac{41}{10}\right)  }{10} \right) \delta^{n-1} \\ 
&- \dfrac{3}{2} \hat{x}^{n} + \left( \dfrac{ 35\Delta t + 10 }{10 \Delta t} - \dfrac{41}{10} \right)\delta^{n}.  \nonumber
\end{align}

\subsection{Model selection}
As displayed in \cref{fig:3d_linear_lasso_path_coef,fig:3d_linear_lasso_path_nz}, by applying the model with $p=1$, $k=1$ and $\lambda$ chosen as $10^{-12}$ from the Pareto front of \emph{lasso} path, the resulting model is found to only contain 4 non-zero terms. 

\begin{figure}[H]
	\centering
	\includegraphics[scale=0.15]{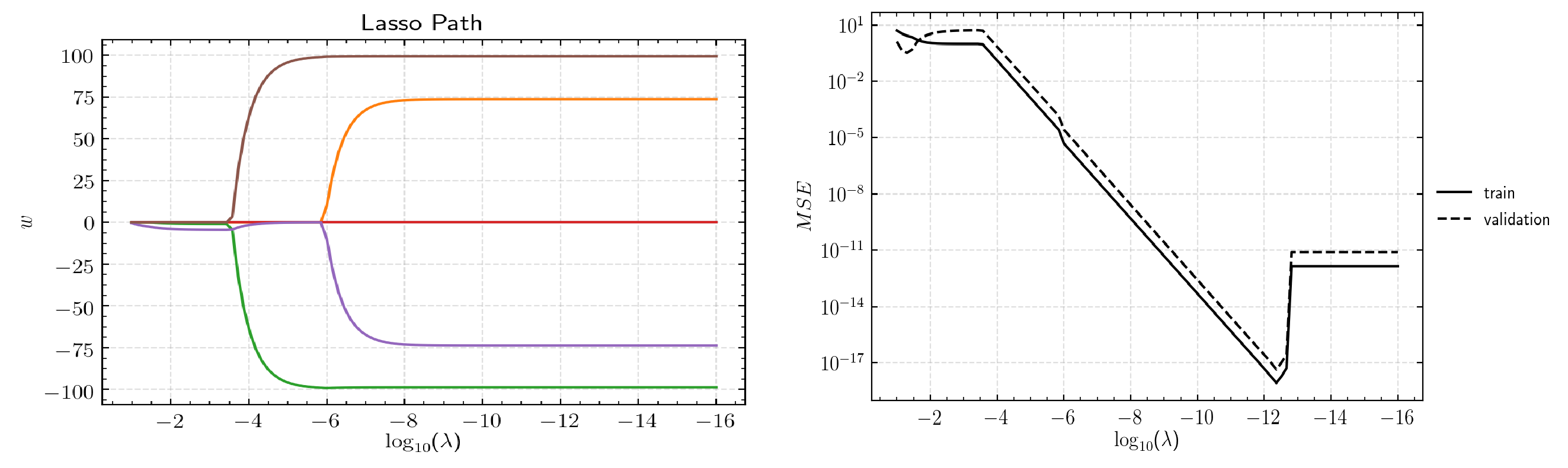}
	\caption{ lasso path for the 3D linear system. Left: coefficients. Right: MSE.}
	\label{fig:3d_linear_lasso_path_coef}
\end{figure}

\begin{figure}[H]
	\centering
	\includegraphics[scale=0.5]{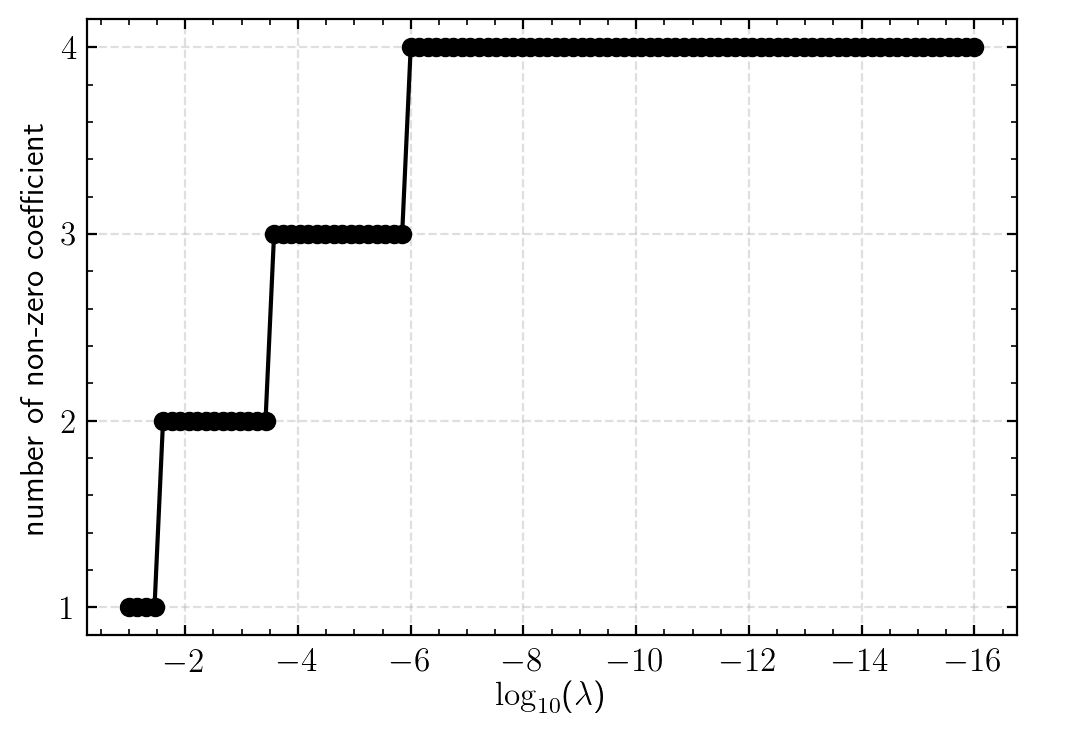}
	\caption{lasso path for the 3D linear system: number of non-zero terms}
	\label{fig:3d_linear_lasso_path_nz}
\end{figure}

\subsection{A posteriori evaluation of model performance}
Using the hyperparameters determined above, the predicted trajectory of $\bm{\hat{x}}(t)$ is found to match the target trajectory to an excellent degree, as shown in \cref{fig:3d_linear_posteriori}. 
\begin{figure}[H]
	\centering
	\includegraphics[scale=0.15]{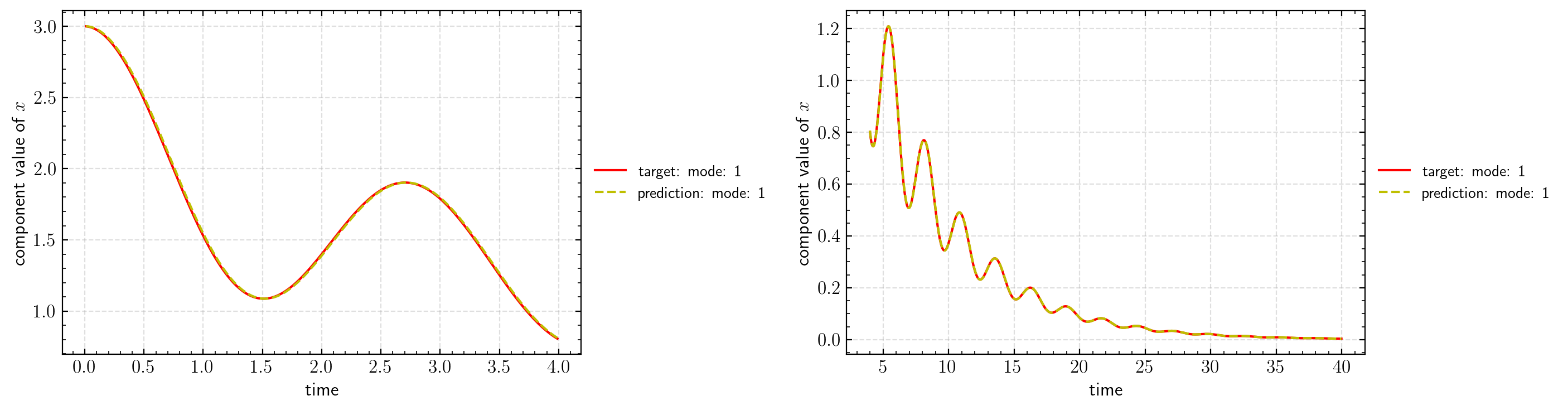}
	\caption{A posteriori model performance on the linear system. Left: training data. Right: testing data.}
	\label{fig:3d_linear_posteriori}
\end{figure}

However, if one sets $p=0$, the resulting model cannot produce good predictions, as the true solution is  $\delta^{n+1} - \delta^{n} = -\Delta t(1.5 \hat{x}^{n} + 4.1\delta^{n} + 6.1x_3^{n})$. Since $x_3^{n}$ is unknown to $\{ \hat{x}^{n}, {\delta}^{n} \}$,  additional memory length is required. The \emph{lasso} path is shown in \cref{fig:3d_linear_lasso_path_coef_,fig:3d_linear_lasso_path_nz_} for the case with insufficient memory. The corresponding a posteriori performance is shown to be poor in \cref{fig:3d_linear_posteriori_p0}. 
\begin{figure}[H]
	\centering
	\includegraphics[scale=0.15]{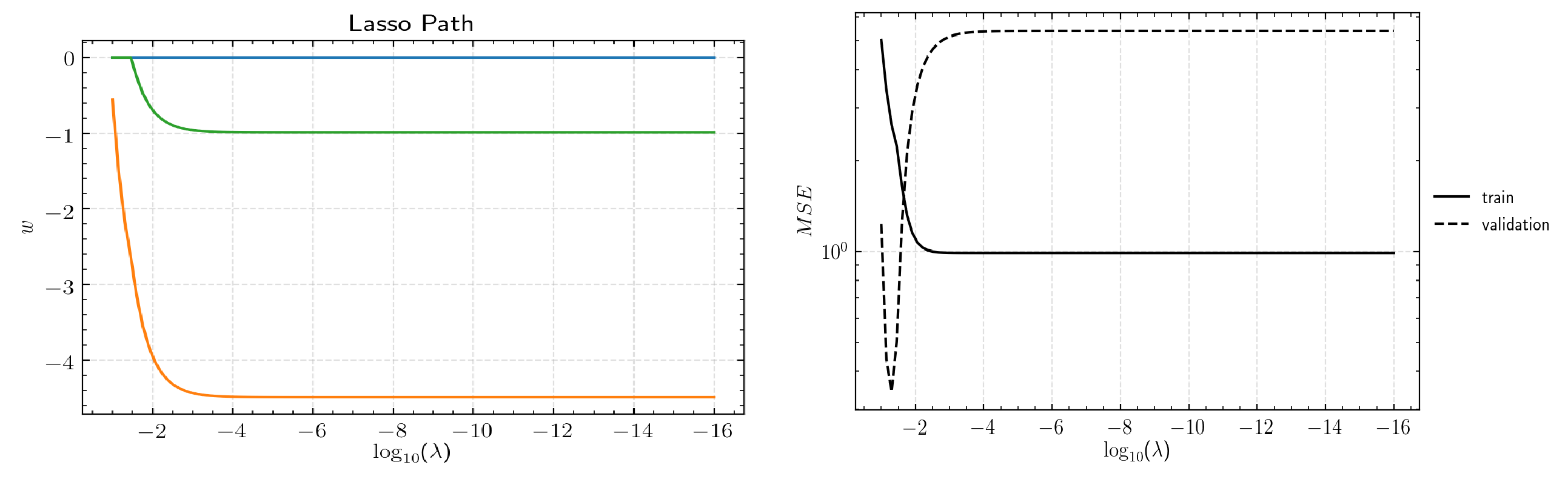}
	\caption{lasso path for the 3D linear system. Left: coefficients. Right: MSE. }
	\label{fig:3d_linear_lasso_path_coef_}
\end{figure}

\begin{figure}[H]
	\centering
	\includegraphics[scale=0.5]{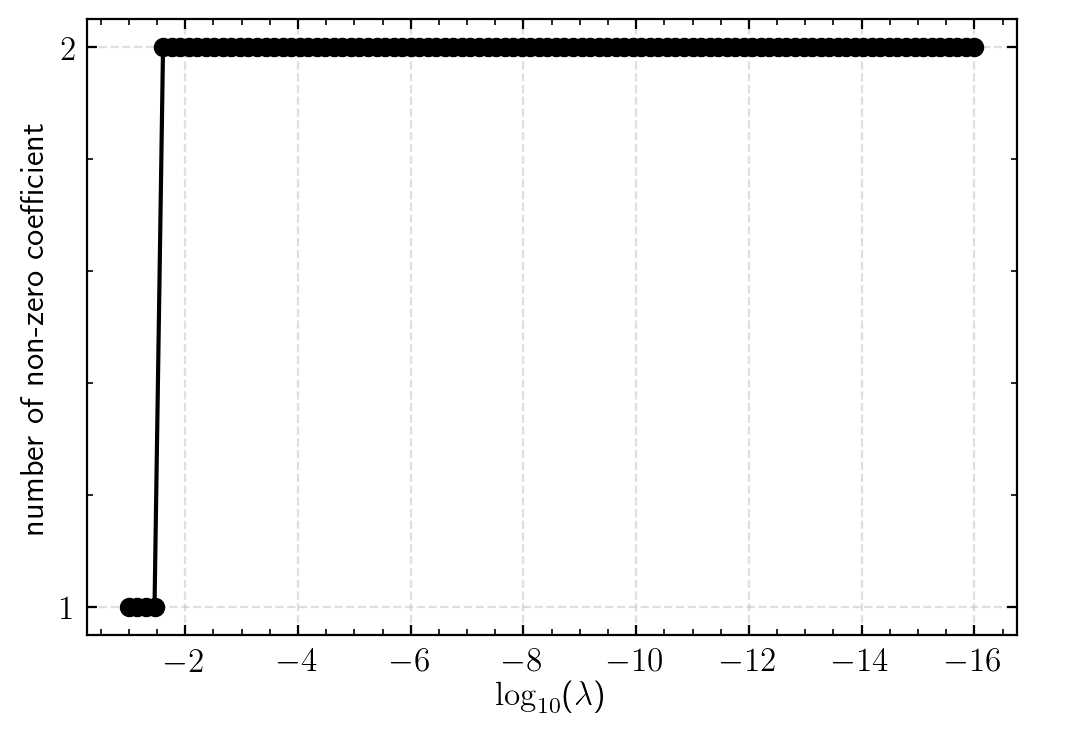}
	\caption{lasso path for the 3D linear system: number of non-zero terms}
	\label{fig:3d_linear_lasso_path_nz_}
\end{figure}

\begin{figure}[H]
	\centering
	\includegraphics[scale=0.15]{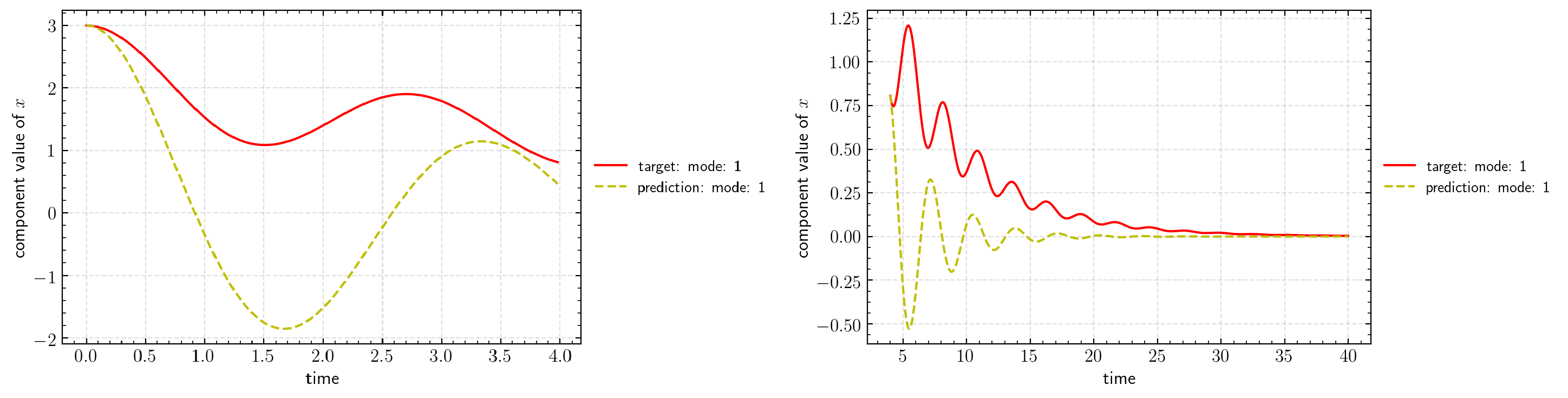}
	\caption{A posteriori model performance on the linear system with $p=0$. Left: training data. Right: testing data.}
	\label{fig:3d_linear_posteriori_p0}
\end{figure}


One might suspect that the sparsest solution, i.e., $\min \lVert \bm{w}_k \rVert_0$, should  contain at most 3 non-zero terms because as mentioned before, there is one redundancy in $\{{x}_1^{n}, {x}_1^{n-1}, {\delta}^{n-1}\}$. The \emph{lasso} based on the $\ell_1$ norm does not guarantee the sparsest solution in the sense of the $\ell_0$ norm, but it makes the problem computationally tractable. 

\section{Theoretical results} \label{sec:theory}

In this section, theoretical results are presented with regard to the capability of the closure model to determine the resolved dynamics with time-delayed features.

\begin{definition}[Nonlinear dynamical system with dual linear closure]{\label{def:general_ds_dual_linear}}
A nonlinear dynamical system with dual linear closure is defined in the following form:
$$
\dfrac{d}{dt}
\begin{bmatrix}
\textcolor{black}{\bm{\hat x}}        \\
\textcolor{black}{\bm{\tilde x}}
\end{bmatrix}
=
\begin{bmatrix}
\bm{F}(\bm{\hat{x}}) +  A_{12}\bm{\tilde{x}}\\
\bm{H}(\bm{\hat{x}}) +  A_{22}\bm{\tilde{x}}
\end{bmatrix},
$$
where $\bm{\hat{x}} \in \mathbb{R}^{Q}$, $\bm{\tilde{x}} \in \mathbb{R}^{N-Q}$,  $\bm{F(\cdot)}: \mathbb{R}^{Q}  \mapsto \mathbb{R}^{Q}$ and $\bm{H(\cdot)}: \mathbb{R}^{Q} \mapsto \mathbb{R}^{N-Q}$, $A_{12} \in \mathbb{R}^{Q\times {N-Q}}$, $A_{22} \in \mathbb{R}^{N-Q \times N-Q}$ with $\bm{\delta} = A_{12}\bm{\tilde{x}}$, $N \in \mathbb{N}$ and $Q \in \mathbb{N}$, $Q<N$.

The corresponding first order forward discretized dynamical system is
$$
\begin{bmatrix}
\textcolor{black}{\bm{\hat x}^{n+1}}        \\
\textcolor{black}{\bm{\tilde x}^{n+1}}
\end{bmatrix}
=
\begin{bmatrix}
\textcolor{black}{\bm{\hat x}^{n}}        \\
\textcolor{black}{\bm{\tilde x}^{n}}
\end{bmatrix}
+
\Delta t
\begin{bmatrix}
\bm{F}(\bm{\hat{x}}^{n}) +  A_{12}\bm{\tilde{x}}^{n}\\
\bm{H}(\bm{\hat{x}}^{n}) +  A_{22}\bm{\tilde{x}}^{n}
\end{bmatrix},
$$
where $n$ denotes steps in time.
\end{definition}

The exact discrete closure dynamics is 
$$
\bm{\delta}^{n+1} = A_{12}\bm{\tilde{x}}^{n+1} = \bm{\delta}^{n} + \Delta t (A_{12}A_{22} \bm{\tilde{x}}^{n}  + A_{12}\bm{H}(\bm{\hat{x}}^{n})).
$$

Note that the only unknown is $\bm{\tilde{x}}^{n}$. Clearly, a sufficient condition for the data-driven framework to exactly represent the closure term $A_{12}\bm{\tilde{x}}^{n}$, would be the determination of $\bm{\tilde{x}}^{n}$ from a linear combination of $\bm{\hat{x}}^{n-1},\ldots,\bm{\hat{x}}^{n-p}$, $\bm{\delta}^n,\ldots,\bm{\delta}^{n-p}$. It will be shown that, with certain structures of $A_{12}$ and $A_{22}$, one can recover the entire history of $\bm{\tilde{x}}$ using up to previous $p$ steps. The following proposition has strong similarities with the observability problem in linear system theory.

\begin{theorem}{\label{thm:recover_linear}} 
For $k \in \mathbb{N}^{+}$, define the following matrix $\mathcal{O}_{k} \in \mathbb{R}^{ kQ \times (N-Q)}$
$$
\mathcal{O}_{k} = \begin{bmatrix}
A_{12} \\
A_{12}A_{22} \\
\vdots \\
A_{12} A_{22}^{k-1} 
\end{bmatrix},
$$
and the following mapping $r_{\mathcal{O}}(\cdot): \mathbb{N}^{+} \mapsto \mathbb{N}$

$$r_{\mathcal{O}}(k)=\rank(\mathcal{O}_{k}).$$ 
If $\mathcal{O}_{N-Q}$ is full rank, i.e., $r_{\mathcal{O}}(N-Q) = N-Q$
, then for the first order forward discretized system with dual linear closure, $\exists p, n \in \mathbb{N}$, with collected $\bm{\delta} \in \mathbb{R}^{Q}$ and $\bm{\hat{x}} \in \mathbb{R}^{Q}$ up to step $n$, such that $ \bm{\tilde{x}}^{n}, \ldots, \bm{\tilde{x}}^{n-p}$ can be determined as a linear combination of $\bm{H}(\bm{\hat{x}}^{n-1}),\ldots,\bm{H}(\bm{\hat{x}}^{n-p})$, $\bm{\delta}^n,\ldots,\bm{\delta}^{n-p}$. For $p=0$, only $\bm{\delta}^{n}$ is used. 

Further, the minimal $p_{*}$ that satisfies the above is $$p_{*} = \min \Omega_{\mathcal{O}}
-1, $$
where $$
\Omega_{\mathcal{O}} = \{ l | l\in \mathbb{N}^{+}, r_{\mathcal{O}}(l) = N-Q \}.
$$

\end{theorem}

\begin{proof}
Consider the first order forward discretized system of a dynamical system with linear closure, with $n, p \in \mathbb{N}$, $n > p$. We have the following \emph{evolution} equations for the unresolved variable $\bm{\tilde{x}}$
\begin{gather}
\bm{\tilde{x}}^{n} = (I+\Delta t A_{22})\bm{\tilde{x}}^{n-1} + \Delta t \bm{H}( \bm{\hat{x}}^{n-1}),\\
\cdots \nonumber \\ 
\bm{\tilde{x}}^{n-p+1} = (I+\Delta t A_{22})\bm{\tilde{x}}^{n-p} + \Delta t \bm{H} ( \bm{\hat{x}}^{n-p}),
\end{gather}
and \emph{projection} equations for $\bm{\delta}$
\begin{gather}
A_{12}\bm{\tilde{x}}^{n} = \bm{\delta}^{n},  \\
\cdots \nonumber \\
A_{12}\bm{\tilde{x}}^{n-p} = \bm{\delta}^{n-p}.  
\end{gather}
Note that the independent unknowns are $\{ \bm{\tilde{x}}^{n-i} \}_{i=0}^{p}$ and we are provided with $\{ \bm{\hat{x}}^{n-i} \}_{i=1}^{p}$ and $\{ \bm{\delta}^{n-i} \}_{i=0}^{p}$. 
Rearranging equations in matrix form, we have 
\begin{equation}{\label{eq:thm_linear_system}}
\bm{\Gamma}_p \bm{X}_p
= \bm{\Sigma}_p,
\end{equation}
where \begin{equation}
\bm{\Gamma}_{p} = \begin{bmatrix}
I & -(I+\Delta t A_{22})  & \ldots & \bm{0} \\
\bm{0} & I & -(I+\Delta t A_{22}) & \ldots  \\
\vdots & \vdots & \vdots & \vdots \\
\bm{0} & \ldots & I & -(I+\Delta t A_{22})  \\
A_{12} & \bm{0} & \bm{0} & \ldots \\
\bm{0} & A_{12} & \bm{0} & \ldots \\
\vdots & \vdots & \vdots & \vdots \\
\bm{0} & \bm{0} &  \ldots  & A_{12}
\end{bmatrix},
\end{equation}
\begin{equation}
\bm{X }_p = \begin{bmatrix}
\bm{\tilde{x}}^{n} \\
\bm{\tilde{x}}^{n-1} \\
\vdots \\
\vdots \\
\vdots \\
\vdots \\
\bm{\tilde{x}}^{n-p} 
\end{bmatrix} ,
\qquad
\bm{\Sigma}_p = \begin{bmatrix}
\Delta t \bm{H}( \bm{\hat{x}}^{n-1}) \\
\Delta t \bm{H}(\bm{\hat{x}}^{n-2}) \\
\vdots \\
\Delta t \bm{H} (\bm{\hat{x}}^{n-p}) \\
\vdots \\
\bm{\delta}^{n} \\
\vdots \\
\bm{\delta}^{n-p}
\end{bmatrix},
\end{equation}

Using row operations to remove the diagonal block matrix of $A_{12}$, 

\begin{align}
\bm{\Gamma}_{p} \rightarrow \begin{bmatrix}
I & -(I+\Delta t A_{22})  & \ldots & \bm{0} \\
\bm{0} & I & -(I+\Delta t A_{22}) & \ldots  \\
\vdots & \vdots & \vdots & \vdots \\
\bm{0} & \bm{0} & I & -(I+\Delta t A_{22})  \\
\bm{0} & \bm{0} &  \ldots  & A_{12}(I+\Delta t A_{22})^p \\
\bm{0} & \bm{0} & \ldots & A_{12}(I+\Delta t A_{22})^{p-1} \\
\vdots & \vdots & \vdots & \vdots \\
\bm{0} & \bm{0} &  \ldots  & A_{12}
\end{bmatrix} ,
\end{align}

Thus 
\begin{equation}
\rank (\bm{\Gamma}_p) = p(N-Q) + \rank(\begin{bmatrix}
A_{12}(I + \Delta t A_{22})^p \\
A_{12}(I + \Delta t A_{22})^{p-1} \\
\vdots \\
A_{12}
\end{bmatrix}).
\end{equation}

Note that 
\begin{equation}
\rank(\begin{bmatrix}
A_{12}(I + \Delta t A_{22})^p \\
A_{12}(I + \Delta t A_{22})^{p-1} \\
\vdots \\
A_{12}
\end{bmatrix} = 
\rank(\begin{bmatrix}
A_{12} A_{22}^p \\
A_{12} A_{22}^{p-1} \\
\vdots \\
A_{12}
\end{bmatrix}) = \rank (\mathcal{O}_{p+1}).
\end{equation}

From basic linear algebra, it is known that $r_{\mathcal{O}}(\cdot)$  is bounded and monotonic, where $r_{\mathcal{O}}(\cdot): \mathbb{N} \mapsto \mathbb{N}$. Also recall that $\mathcal{O}_{N-Q}$ is full rank thus the following set $\Omega$ is not empty 
\begin{equation}
\Omega_{\mathcal{O}} = \{ l | l\in \mathbb{N}^{+}, r_{\mathcal{O}}(l) = N-Q \}.
\end{equation}

Therefore setting
\begin{equation}
p_{*} =  \min \Omega_{\mathcal{O}}  - 1,
\end{equation} 
we will have 
\begin{equation}
\rank(\bm{\Gamma}_p) = p(N-Q) + (N-Q) = (p+1)(N-Q),
\end{equation}
indicating $\bm{\Gamma}_p$ is full column rank. Therefore, consider $\bm{\Gamma}^{+}_p$ as the \emph{left} Moore-Penrose inverse of $\bm{\Gamma}_{p}$
\begin{equation}
\bm{\Gamma}^{+}_p = (\bm{\Gamma}^\top_p\bm{\Gamma}_p)^{-1}\bm{\Gamma}_p ,
\end{equation}
and thus
\begin{equation}
\bm{X}_p = \bm{\Gamma}^{+}_p \bm{\Gamma}_p \bm{X}_p = \bm{\Gamma}^{+}_p\bm{\Sigma}_p.
\end{equation}



\end{proof}

Again, note that once $\bm{\tilde{x}}^{n}$ is determined from past time instances of $\bm{\delta}$ and $\bm{\hat{x}}$ up to step $p$, the closure dynamics is fully determined as well. As an example, applying \cref{thm:recover_linear} to the 3D discrete linear system described in \cref{eq:3d_linear_system}, 
$A_{12} = \begin{bmatrix}
-1 & -1
\end{bmatrix}$ and 
$A_{22} = \begin{bmatrix}
- 1.1 & 1.5 \\
- 3 & 0.5
\end{bmatrix}$, $\rank( A_{12} ) = 1$, $\rank(\begin{bmatrix}
A_{12} \\
A_{12} A_{22}
\end{bmatrix}) = \rank( \begin{bmatrix} -1 & -1 \\ 4.1 & -2 \end{bmatrix} ) = 2$. Therefore, $p_{*} = \min \Omega_{\mathcal{O}} -1 = 1$.
As a trivial observation, based on the \cref{thm:recover_linear}, one can immediately obtain the following proposition.
\begin{proposition}
If $A_{12}$ has full column rank, $\bm{\tilde{x}}^{n}$ can be determined with only $\bm{\delta}^{n}$.
\end{proposition}

However, full observability on all past states of $\bm{\tilde{x}}$ is a very strong criterion to guarantee predictability of dual linear closure dynamics. Indeed, from a data-driven perspective, one only requires that the linear closure is in the $p$-\emph{time delayed} observable space of $\bm{\tilde{x}}$ as defined in  \cref{def:time_delay_obsv}. Therefore, we now turn our focus to finding $A_{12}\bm{\tilde{x}}^{n}$ directly. First, a strict definition of $p$-time delayed linear observable space is given below.

\begin{definition}[$p$-time delayed linear observable space]{\label{def:time_delay_obsv}}
For the first order forward discretized nonlinear dynamical system with dual linear closure, define the corresponding $p$ time delayed linear observable space $\chi_{p}$ as
$$
\chi_p = \{  \eta | \eta = v^\top \bm{X}_p,  v \in \Ima V_p   \},
$$
where $V_p$ is from the reduced singular value decomposition of $\bm{\Gamma}_p$ 
$$\bm{\Gamma}_{p} = U_pS_pV_p^\top,$$
with $U_p \in \mathbb{R}^{(pN+Q)\times r}$, $S_p \in \mathbb{R}^{r \times r}$, $V_p \in \mathbb{R}^{(p+1)(N-Q) \times r}$, $r = \rank(\bm{\Gamma}_p)$.
\end{definition}

Regarding the question of determining a general linear combination of $\bm{X}_p$ from a $p$-time delayed observable space, the following lemma shows that a rank test can provide essential insight.
\begin{lemma}{\label{lem:partial_obsv}}
For a nonlinear dynamical system with dual linear closure, for any quantity $\xi \in \mathbb{R}^{q \times 1}$ that is a linear combination of $\bm{X}_p$ characterized by $C$, 
$$
\xi = C^\top \bm{X}_p,
$$
where $C \in \mathbb{R}^{ (p+1)(N-Q) \times  q}$, if 
$$
\rank(V_p) = \rank(\begin{bmatrix}
C^\top \\
V_p^\top 
\end{bmatrix}),
$$
then
$$
\xi \in \chi_p,$$
i.e., $\xi$ is observable with $p$-time delayed information of $\bm{\delta}$ and $\bm{\hat{x}}$.
\end{lemma}
\begin{proof}
$\because \rank(\begin{bmatrix}
C^\top \\
V_p^\top 
\end{bmatrix}) = \rank(\begin{bmatrix}
C & 
V_p 
\end{bmatrix}) = \rank(V^\top_p) = \rank(V_p)$ $\therefore C \subset \Ima V_p.$ 

$\therefore \xi = C^\top \bm{X}_p \in \chi_p.$ 
\end{proof}

Given the above lemma, one can obtain a rank test criterion in \cref{thm:rank_test_dual_linear} for whether the closure dynamics of a nonlinear system with dual linear closure can be determined with $p$ time delayed observable space. Furthermore, analysis of the rank of the augmented matrix provides further insights into the role of time delay in observation.

\begin{theorem}{\label{thm:rank_test_dual_linear}}
A nonlinear dynamical system with dual linear closure with $p = N-Q-1$ will satisfy the following rank test
\begin{gather*}
\rank(V_p) = \rank(\begin{bmatrix}
C^\top \\ 
V_p^\top 
\end{bmatrix}),
\end{gather*}
where $C^\top = \begin{bmatrix}
A_{12}A_{22} & \bm{0}
\end{bmatrix}$, and the closure dynamics is observable from $p$ time delayed observable space, i.e., can be determined as a linear combination of $\bm{H}(\bm{\hat{x}}^{n-1}),\ldots,\bm{H}(\bm{\hat{x}}^{n-p})$, $\bm{\delta}^n,\ldots,\bm{\delta}^{n-p}$. 
Furthermore, the minimal number of previous steps $p$ that satisfies the above condition can be found through
$$
p_{*} = \min \Pi_{\mathcal{O}} - 1,
$$
where
$$\Pi_{\mathcal{O}} = \{ l | l \in \mathbb{N}^{+}, r_{\mathcal{O}}(l) = r_{\mathcal{O}}(l+1) \}.
$$
\end{theorem}
\begin{proof}
To determine $\bm{\delta}^{n+1}$, $A_{12}A_{22}\bm{\tilde{x}}^{n}$ has to be in the $p$ time delayed observable space.
$\because \bm{\tilde{x}}^{n} =  \begin{bmatrix}
A_{12}A_{22} & \bm{0}
\end{bmatrix} \bm{X}_p. \therefore$ from \cref{lem:partial_obsv}, if 
\begin{gather*}
\rank(V_p) = \rank(\begin{bmatrix}
A_{12}A_{22} \quad \bm{0} \\ 
V_p^\top 
\end{bmatrix}),
\end{gather*}
then $\bm{\tilde{x}}^{n}$ is $p$ time delayed linear observable. Since $V_p^\top$ shares the same independent row space as $\bm{\Gamma}$, augmenting $\bm{\Gamma}_p$ with $C^\top$ will result in the same rank. 
\begin{align*}
\begin{bmatrix}
\bm{\Gamma}_p \\
C^\top
\end{bmatrix}
& = 
\begin{bmatrix}
I & -(I+\Delta t A_{22})  & \ldots & \bm{0} \\
\bm{0} & I & -(I+\Delta t A_{22}) & \ldots  \\
\vdots & \vdots & \vdots & \vdots \\
\bm{0} & \ldots & I & -(I+\Delta t A_{22})  \\
A_{12} & \bm{0} & \bm{0} & \ldots \\
\bm{0} & A_{12} & \bm{0} & \ldots \\
\vdots & \vdots & \vdots & \vdots \\
\bm{0} & \bm{0} &  \ldots  & A_{12} \\
A_{12}A_{22} & \bm{0} &  \ldots & \bm{0}
\end{bmatrix} \\
& \rightarrow 
\begin{bmatrix}
I & -(I+\Delta t A_{22})  & \ldots & \bm{0} \\
\bm{0} & I & -(I+\Delta t A_{22}) & \ldots  \\
\vdots & \vdots & \vdots & \vdots \\
\bm{0} & \bm{0} & I & -(I+\Delta t A_{22})  \\
\bm{0} & \bm{0} &  \ldots  & A_{12}(I+\Delta t A_{22})^p \\
\bm{0} & \bm{0} & \ldots & A_{12}(I+\Delta t A_{22})^{p-1} \\
\vdots & \vdots & \vdots & \vdots \\
\bm{0} & \bm{0} &  \ldots  & A_{12} \\
\bm{0} & \bm{0} &  \ldots  & A_{12}A_{22}(I+\Delta t A_{22})^p
\end{bmatrix} \\
& \rightarrow \rank(\begin{bmatrix}
\bm{\Gamma}_p \\
C^\top
\end{bmatrix} ) = p(N-Q) + \rank(\begin{bmatrix}
A_{12} A_{22}^p \\
A_{12} A_{22}^{p-1} \\
\vdots \\
A_{12} \\
A_{12} A_{22}^{p+1}
\end{bmatrix}) = p(N-Q) + \rank(\mathcal{O}_{p+2}) \\ 
& \rightarrow \rank(\begin{bmatrix}
\bm{\Gamma}_p \\
C^\top
\end{bmatrix} ) = \rank(\bm{\Gamma}_p ) \rightarrow \rank(\mathcal{O}_{p+2}) = \rank(\mathcal{O}_{p+1}).
\end{align*}
$\because$ Recall $r_{\mathcal{O}}(\cdot)$ is a monotonic integer function and from the Cayley Hamilton theorem, $A_{22}^{N-Q}$ is linearly dependent on $\{ I, A_{22},\ldots, A_{22}^{N-Q-1} \} $ $\therefore \forall p \ge N-Q-1$, $\rank(\mathcal{O}_{p+2}) = \rank(\mathcal{O}_{p+1})$.
Correspondingly the minimal number of previous steps that satisfies the rank test can be defined as the minimal integer that satisfies the $\rank(\mathcal{O}_{p+2}) = \rank(\mathcal{O}_{p+1})$. 
\begin{equation}
p_{*} = \min \Pi_{\mathcal{O}} - 1,
\end{equation}
where
\begin{equation}
\Pi_{\mathcal{O}} = \{ l | l \in \mathbb{N}^{+}, r_{\mathcal{O}}(l) = r_{\mathcal{O}}(l+1) \}.
\end{equation}
Because of the monotonicity of integer function $r_{\mathcal{O}}(\cdot)$, $p_{*}$ can be found in a sequential sense.
\end{proof}

The fact that one can determine the closure dynamics of any nonlinear system with dual linear closure given $\emph{all}$ previous resolved states is not surprising. It will be shown shortly that this is possible for a slightly more general case. 
\begin{proposition}
Closure dynamics of any nonlinear dynamical system with dual linear closure can be determined as a linear combination of $\bm{\hat{x}}^{n-1},\ldots,\bm{\hat{x}}^{n-p}$, $\bm{\delta}^n,\ldots,\bm{\delta}^{n-p}$, with $p \le N-Q-1$.
\end{proposition}

As a trivial observation, if we replace closure with $\bm{\tilde{x}}^{n}$, one can easily obtain the following rank test as a criterion.
\begin{proposition}
For a nonlinear dynamical system with dual linear closure, if
\begin{gather*}
\rank(V_p) = \rank(\begin{bmatrix}
C^\top \\ 
V_p^\top 
\end{bmatrix}),
\end{gather*}
where $C^\top = \begin{bmatrix}
I_{N-Q \times N-Q} & \bm{0}
\end{bmatrix}$,
then $\bm{\tilde{x}}^{n}$ is observable from a $p$ time delayed observable space. 
\end{proposition}

The key ingredient of \cref{thm:recover_linear} is the exploitationo of \emph{projection} equations in the \emph{dual linear} closure setting, which may be overlooked since they share the same information as previous observables in the statistical sense if the initial condition is fully known. Without the explicit usage of \emph{projection} equations, one can obtain a closure with explicit memory dependence on \emph{all} previous observables, but is correspondingly applicable to a more general system stated in \cref{def:general_ds_linear_closure}.

\begin{definition}[Nonlinear dynamical system with linear closure]{\label{def:general_ds_linear_closure}}
A nonlinear dynamical system with linear closure is defined as
\begin{equation}
\dfrac{d}{dt}
\begin{bmatrix}
\textcolor{black}{\bm{\hat x}}        \\
\textcolor{black}{\bm{\tilde x}}
\end{bmatrix}
=
\begin{bmatrix}
\bm{F}(\bm{\hat{x}},\bm{\tilde{x}})\\
\bm{H}(\bm{\hat{x}}) +  A_{22}\bm{\tilde{x}}
\end{bmatrix},
\end{equation}
where $\bm{\hat{x}} \in \mathbb{R}^{Q}$, $\bm{\tilde{x}} \in \mathbb{R}^{N-Q}$,  $\bm{F(\cdot)}: \mathbb{R}^{N} \mapsto \mathbb{R}^{Q}$ and $\bm{H(\cdot)}: \mathbb{R}^{Q} \mapsto \mathbb{R}^{N-Q}$, $A_{22} \in \mathbb{R}^{N-Q \times N-Q}$ with $\bm{\delta} = A_{12}\bm{\tilde{x}}$ and $Q \in \mathbb{N}$, $Q<N$.
\end{definition}

\begin{corollary}
With only evolution equations, one can write the following equation for a first order forward discretized dynamical system
$\forall n, p \in \mathbb{N}$, $n > p$
$$ 
\bm{\tilde{x}}^{n} = (I+\Delta t A_{22})^{p} \bm{\tilde{x}}^{n-p} + \sum_{l=0}^{p-1}\Delta t (I + \Delta t A_{22})^{l} \bm{H}( \bm{\hat{x}}^{n-l-1}),
$$  
which links the unresolved states between any two time instances.
\end{corollary}

\begin{proof}
Considering only the \emph{evolution} equations, one can write the following in matrix form 
\begin{equation}{\label{eq:thm_linear_system_mz}}
\bm{\Gamma}^{e}_p \bm{X}_p
= \bm{\Sigma}^{e}_p,
\end{equation}
where \begin{equation}
\bm{\Gamma}^{e}_{p} = \begin{bmatrix}
I & -(I+\Delta t A_{22})  & \ldots & \bm{0} \\
\bm{0} & I & -(I+\Delta t A_{22}) & \ldots  \\
\vdots & \vdots & \vdots & \vdots \\
\bm{0} & \ldots & I & -(I+\Delta t A_{22})  
\end{bmatrix},
\end{equation}
\begin{equation}
\bm{X }_p = \begin{bmatrix}
\bm{\tilde{x}}^{n} \\
\bm{\tilde{x}}^{n-1} \\
\vdots \\
\vdots \\
\vdots \\
\vdots \\
\bm{\tilde{x}}^{n-p} 
\end{bmatrix}, 
\qquad
\bm{\Sigma}^{e}_p = \begin{bmatrix}
\Delta t \bm{H}( \bm{\hat{x}}^{n-1}) \\
\Delta t \bm{H}(\bm{\hat{x}}^{n-2}) \\
\vdots \\
\Delta t \bm{H} (\bm{\hat{x}}^{n-p}) 
\end{bmatrix}.
\end{equation}

Recall we are interested in $\bm{\hat{x}}^n$, with several row operations on the first row block,

\begin{align}
\bm{\Gamma}^{e}_{p} \rightarrow 
\begin{bmatrix}
I & \bm{0}  & \ldots & -(I+\Delta t A_{22})^{p} \\
\ldots & \ldots & \ldots & \ldots  
\end{bmatrix} ,
\end{align}
and correspondingly
\begin{align}
\bm{\Sigma}^{e}_p \rightarrow 
\begin{bmatrix}
\sum_{l=0}^{p-1}\Delta t (I + \Delta t A_{22})^{l} \bm{H}( \bm{\hat{x}}^{n-l-1}) \\
\vdots 
\end{bmatrix},
\end{align}
Therefore,
\begin{equation}
\bm{\tilde{x}}^{n} = (I+\Delta t A_{22})^{p} \bm{\tilde{x}}^{n-p} + \sum_{l=0}^{p-1}\Delta t (I + \Delta t A_{22})^{l} \bm{H}( \bm{\hat{x}}^{n-l-1}).
\end{equation}
\end{proof}

The implication is that, if $\bm{\tilde{x}}$ is known at one previous time instant, the future of $\bm{\tilde{x}}$ starting from that point is completely determined by $\bm{\hat{x}}$ in a convolutional sense. For example, starting from the initial condition, we have the following result often seen in linear systems theory:

\begin{proposition}
For a nonlinear dynamical system with linear closure, if $\bm{\tilde{x}}^{0}$ is known, one can uniquely determine $\bm{\tilde{x}}^{n}$ in the following 
\begin{equation}
\bm{\tilde{x}}^{n} = (I+\Delta t A_{22})^{n} \bm{\tilde{x}}^{0} + \sum_{l=0}^{n-1}\Delta t (I + \Delta t A_{22})^{l} \bm{H}( \bm{\hat{x}}^{n-l-1}).
\end{equation} 
\end{proposition}

The present framework exploits the fact that although the closure is explicitly based on \emph{all} previous information of the observables, the operator driving this function might only possess a finite memory dependence as indicated in \cref{thm:recover_linear}. Therefore, the essential structure of the closure may be much more compact.

\section{Results - Non-linear systems} \label{sec:non-linear}
In this section, the operator inference framework is used to derive closures for several problems 
ranging from chaotic and non-chaotic
nonlinear ordinary differential equation (ODE) systems to nonlinear partial
differential equations (PDE).

\subsection{Van del Pol system}
\subsubsection{Problem description}
The Van del Pol (VdP) system 
with first order forward discretization is
\begin{equation}{\label{eq:2d_vdp_system_disc}}
    \begin{bmatrix}
    x^{n+1}_1\\ x^{n+1}_2
    \end{bmatrix}
    = 
    \begin{bmatrix}
    x^{n}_1\\ x^{n}_2
    \end{bmatrix}
    + \Delta  t
    \begin{bmatrix}
    x^{n}_2 \\
    \mu (1 - x_1^{n}x_1^{n}) x_2^{n} - x_1^{n}
    \end{bmatrix},
\end{equation}
where $\mu = 2$, $\hat{x}(0) = x_1^{0} = 1$, $\tilde{x}(0) = x_2^{0} = 0$, and $\delta = x_2$. The simulation is run from $t \in [0, 60]$ collecting  $\{ \hat{x}, \delta \}$ as data over 6000 snapshots with a $\Delta t = 0.01$. The first 30\% of data is set as training data and the rest is set for testing.

Consider $\hat{x} =x_1$, $\tilde{x} = x_2$ thus $N =2$ and $Q=1$. Correspondingly, the ROM formulation is given below with linear superposition of multi-time effects assumption
\begin{align}
\hat{x}^{n+1} &= \hat{x}^{n} + \Delta t \delta^n, \\
\delta^{n+1} &= \delta^{n} + \Delta t \sum_{i=0}^{p}\bm{G}_i.
\end{align}

For VdP system, the exact solution for the closure dynamics  with $p=0$ is
\begin{equation}
\dfrac{\delta^{n+1} - \delta^{n}}{\Delta t} = \mu \left( 1-\hat{x}^{n}\hat{x}^{n} \right) \delta^{n} - \hat{x}^{n} = -\hat{x}^n + 2\delta^n -2\hat{x}^n\hat{x}^n \delta^n.
\end{equation}

\subsubsection{Model selection for polynomial regression}
To determine the underlying sparse dynamics, the \emph{lasso} path is computed and presented in \cref{fig:2d_vdp_lasso_path_coef,fig:2d_vdp_lasso_path_nz}. It can be seen that an elbow is present in the error plot as $\lambda$ near $10^{-10}$, where a number of non-zero terms jump above 3 to 9, causing a slight increase in MSE. Thus, the optimal $\lambda$ is chosen as $10^{-10}$ according to the Pareto front. 

\begin{figure}[H]
	\centering
	\includegraphics[scale=0.15]{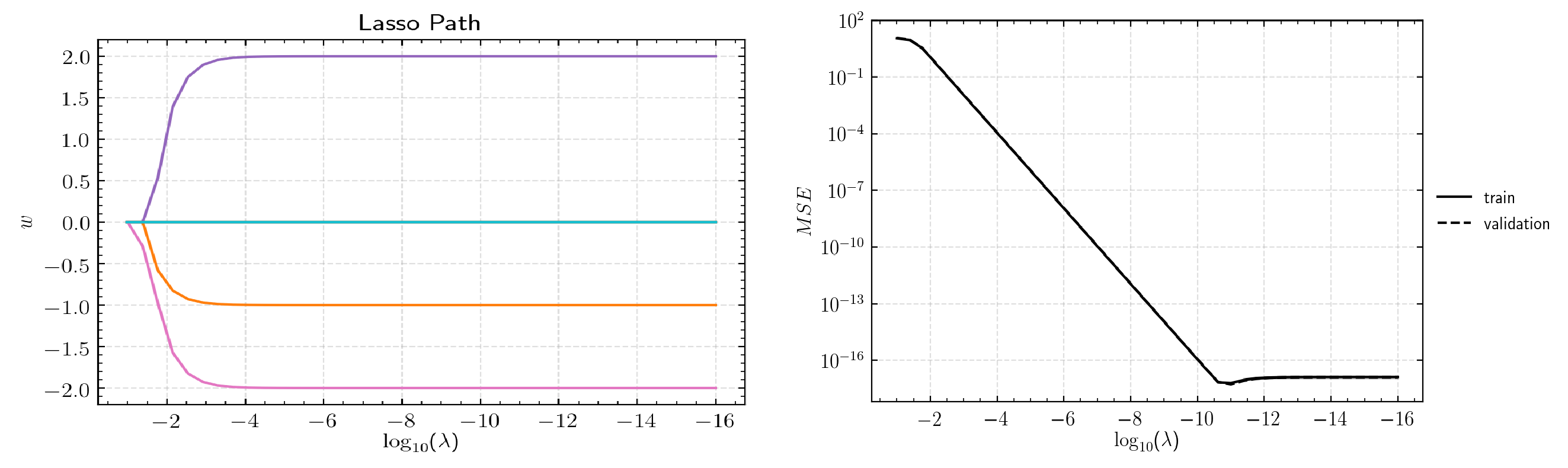}
	\caption{lasso path for 2D VdP system. Left: coefficients. Right: MSE.}
	\label{fig:2d_vdp_lasso_path_coef}
\end{figure}

\begin{figure}[H]
	\centering
	\includegraphics[scale=0.5]{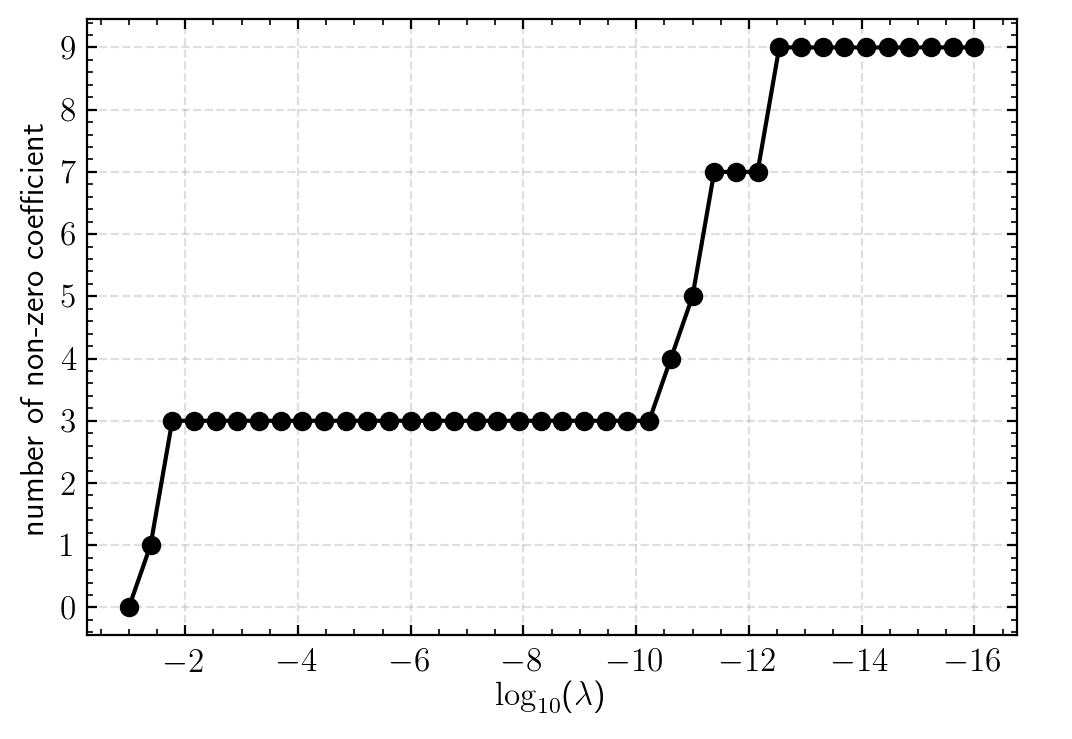}
	\caption{lasso path for 2D VdP system: number of non-zero terms}
	\label{fig:2d_vdp_lasso_path_nz}
\end{figure}

\subsubsection{A posteriori evaluation of model performance}
The corresponding model performance in an a posteriori sense for both training and testing data is excellent, as shown below in \cref{fig:2d_vdp_posteriori}. 

\begin{figure}[H]
	\centering
	\includegraphics[scale=0.15]{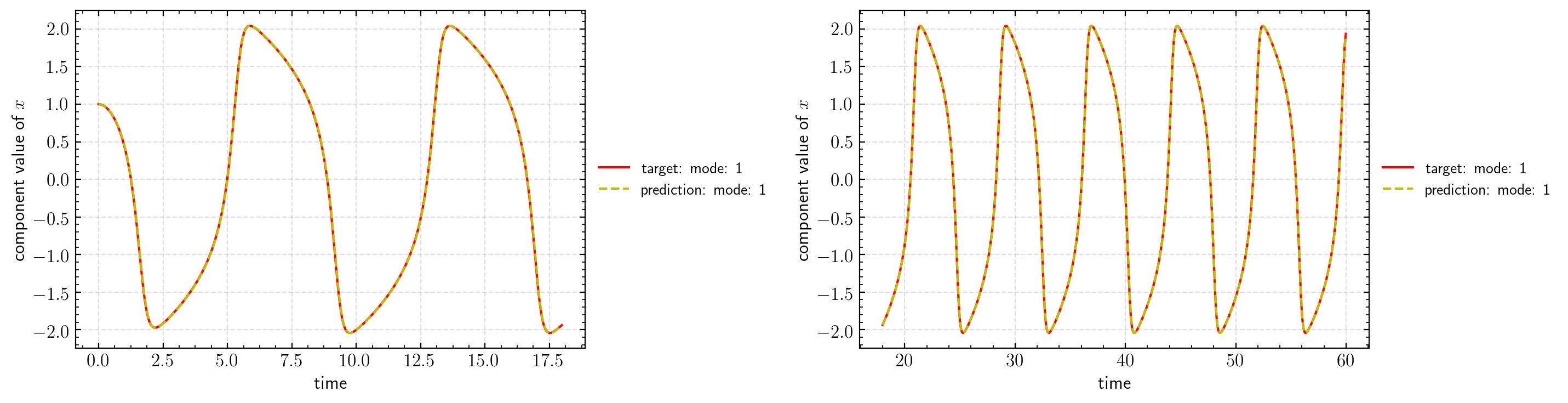}
	\caption{A posteriori model performance on 2D VDP system. Left: training data. Right: testing data.}
	\label{fig:2d_vdp_posteriori}
\end{figure}

\subsection{Duffing Map}
\subsubsection{Problem description}
The Duffing map is a classic example of a chaotic map. We take the form
\begin{align}
x_1^{n+1} &= x_1^{n} + \Delta t (x_2^{n} - x_1^{n}), \\
x_2^{n+1} &= x_2^{n} + \Delta t (-bx_1^{n} + (a-1)x_2^{n} - (x_2^{n})^{3}),
\end{align}
with $a=2.75$ and $b=0.2$, $\Delta t = 1$, $x_1(0) = x_1^{0} = 0.5$, $x_2(0) = x_2^{0} = 0$. The resolved variable $\hat{x} = x_1$. We simulate this system up to 6000 steps with the first 30\% for training, and the rest for testing. For this case, the corresponding closure dynamics for $\delta$ is
\begin{equation}
\delta^{n+1} = a\delta^{n} - (\delta^{n})^3 - b\hat{x}^{n}.
\end{equation}

\subsubsection{Model selection}
As displayed in \cref{fig:2d_duffing_lasso_path,fig:2d_duffing_lasso_path_nz}, by sweeping $\lambda$, an optimal value of $\lambda = 10^{-10}$ is found. At that sparsity level, the resulting expression is given as follows:
\begin{align}
\delta^{n+1} & = \delta^{n} + \Delta t (-0.199999\hat{x}^{n} + 1.749999 \delta^{n} - 0.999999\delta^{n3}  \\ \nonumber & + 2.99\times 10^{-11} \hat{x}^{n2} - 1.34 \times 10^{-8} \hat{x}^{n3}  
  + 2.41\times 10^{-11} \delta^{n}\hat{x}^{n} + 6.81\times 10^{-10} \delta^{n} \hat{x}^{n2}  ).
\end{align}
\begin{figure}[H]
	\centering
	\includegraphics[scale=0.15]{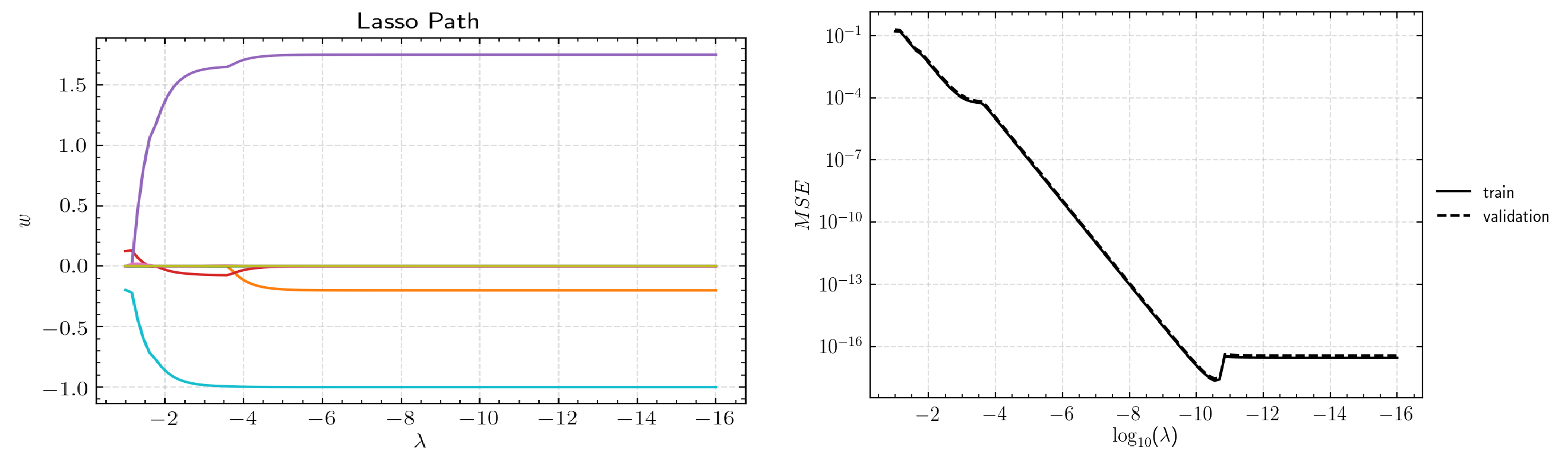}
	\caption{lasso path for 2D Duffing system. Left: coefficient. Right: MSE.}
	\label{fig:2d_duffing_lasso_path}
\end{figure}

\begin{figure}[H]
	\centering
	\includegraphics[scale=0.5]{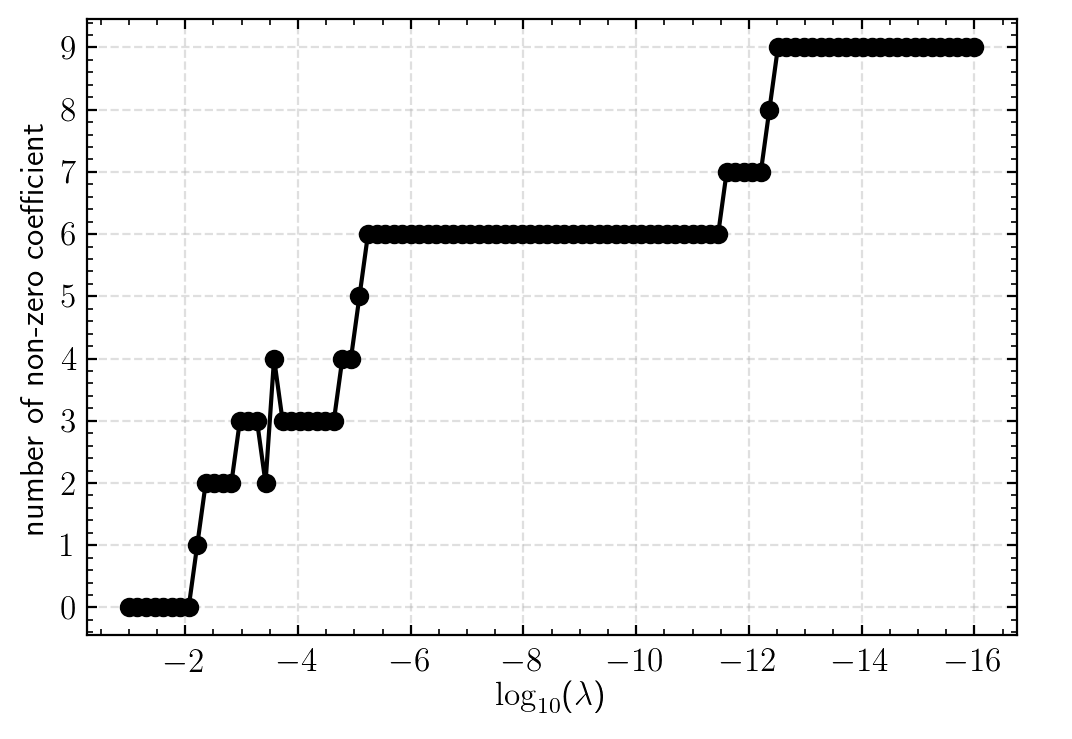}
	\caption{lasso path for 2D Duffing system: number of non-zero terms}
	\label{fig:2d_duffing_lasso_path_nz}
\end{figure}

\subsubsection{A posteriori evaluation of model performance}

For a chaotic system, since it is extremely difficult to achieve accurate long time predictions, models are most often evaluated in a variety of ways. These include subjective visual inspection \cite{han2004prediction} or  measures for the attractor~\cite{aguirre2009modeling} such as maximum Lyapunov exponent \cite{genccay1997nonlinear}, correlation dimension and other time averaged characteristics\cite{lin2003long}. The first approach, although perhaps the most widely used \cite{miyoshi1995learning}\cite{sato1996evolutionary}\cite{trischler2016synthesis}, can sometimes be misleading \cite{diks1996detecting}. 

In this work, we first show there is excellent correspondence in maximum Lyapunov exponent (MLE) and correlation dimension $\gamma$, computed on both the ground truth time series and modeled time series for both training and testing data. Following this, we employ a null hypothesis test proposed by Diks et al.~\cite{diks1996detecting} to show that the attractor reconstructed by embedding the time series predicted by our model is indeed close to the phase space reconstruction of the ground truth within a confidence interval. As suggested by Diks, the null hypothesis that the two delay vectors are drawn from the same multidimensional probability distribution is accepted if $S<3$. Diks criterion has been previously employed as an early stop criterion during the training of neural networks~\cite{Bakker2000}. 

The comparison of the predicted time series between the modeled system and ground truth is displayed in \cref{fig:sindy_2d_duffing_posteriori} for training and testing data. Due to the chaotic nature of the dynamics, direct measurement of the MSE is not suitable for this case. Examination of the MLE and correlation dimension in \cref{tab:df_map_lya_cd} shows excellent agreement. Furthermore, Diks test shows $|S| =  1.003$ for training data and $|S| = 1.588$ for testing data, which further confirms the validity of the model. Details of the implementation of Diks criterion are given in \cref{app:diks_criterion}.

\begin{figure}[H]
	\centering
	\includegraphics[scale=0.15]{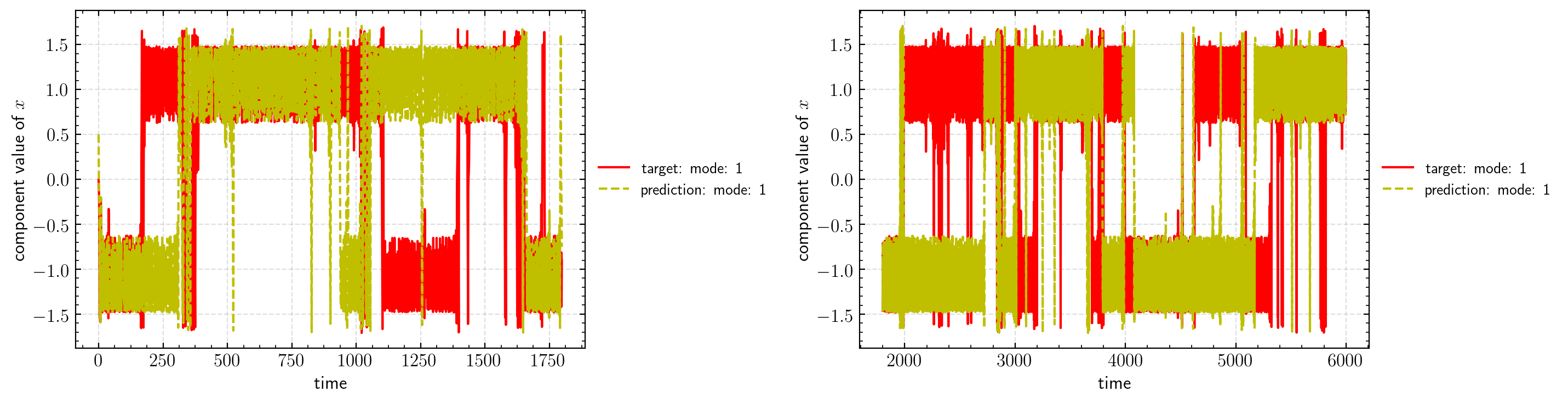}
	\caption{A posteriori model performance on the Duffing map. Left: training data. Right: testing data.}
	\label{fig:sindy_2d_duffing_posteriori}
\end{figure}

\begin{table}[tbhp]
{\footnotesize
\caption{Comparison of MLE and correlation dimension between truth and model}\label{tab:df_map_lya_cd}
\begin{center}
\begin{tabular}{|c|c|c|} \hline
case           & MLE  & $\gamma$  \\ \hline
true train     & 0.98 & 1.12  \\
model train    & 0.97 & 1.12 \\ \hline
true test      & 0.97 & 1.13  \\
model test     & 0.97 & 1.13  \\ \hline
\end{tabular}
\end{center}
} 
\end{table}


\subsection{Lorenz system}
\subsubsection{Problem description}
The corresponding first order forward discretized Lorenz system is given as:
\begin{equation}
\begin{bmatrix}
x_{1}^{n+1} \\ x_{2}^{n+1} \\ x_{3}^{n+1}
\end{bmatrix} = 
\begin{bmatrix}
x_{1}^{n} + \Delta t \sigma(x_{2}^{n}-x_{1}^{n})\\
x_{2}^{n} + \Delta t (x_{1}^{n}(\rho - x_{3}^{n}) -x_{2}^{n})\\
x_{3}^{n} + \Delta t (x_{1}^{n}x_{2}^{n}-\beta x_{3}^{n}),
\end{bmatrix}
\end{equation}
with $x_1(0)=0.5$, $x_2(0)=x_3(0)=0$.
Parameters for each case are shown in \cref{tab:lsys} where the only difference is $\rho$.
\begin{table}[tbhp]
{\footnotesize
\caption{Parameters of Lorenz system for chaotic and nonchaotic cases}\label{tab:lsys}
\begin{center}
\begin{tabular}{|c|c|c|c|} \hline
case & \bf $\sigma$ & $\beta$ & $\rho$ \\ \hline
non-chaotic & 10 & 8/3 & 15 \\
chaotic     & 10 & 8/3 & 35 \\ \hline
\end{tabular}
\end{center}
} 
\end{table}

For the non-chaotic case, the simulation time is $t = [0,20]$ with 8000 snapshots; and for the chaotic case, the simulation time is $t = [0, 400]$ with 40000 snapshots. The snapshots are equally split between training and testing sets.

For the Lorenz system with $\hat{x} = x_1$, $\delta = \sigma x_2$, one can find the analytical closure for $\delta$ with $p=1$ after some algebra: 
\begin{equation}{\label{eq:analytic_form_lorenz}}
\delta^{n+1} = (1 - \Delta t)\delta^{n} + \sigma \Delta t \hat{x}^{n} \left( \left( 1-\beta \Delta t\right )\left (\dfrac{\delta^{n} + (\Delta t - 1)\delta^{n-1}}{\sigma \hat{x}^{n-1} \Delta t} \right) - \dfrac{\hat{x}^{n-1}\delta^{n-1}}{\sigma} + \rho \beta \Delta t \right),
\end{equation}
which clearly involves \emph{cross time} features and \emph{rational} forms instead of pure polynomial forms. 

The corresponding ROM formulation is given as
\begin{align}
\hat{x}^{n+1} &= \hat{x}^{n} - \Delta t \sigma \hat{x}^{n} +  \Delta t \delta^n, \\
\delta^{n+1} &= \delta^{n} + \Delta t \bm{G}(\hat{x}^{n}, \hat{x}^{n-1}, \delta^{n}),
\end{align}
where $\bm{G}$ is modeled by a neural network.

Standard polynomial regression is found to be unsuitable to extract governing equations in this case. A recently developed method called implicit-SINDy~\cite{mangan2016inferring}, which can account for non-rational functions could perhaps improve predictions. Alternatively, we employ an artificial neural network model with $p=1$ and consider full memory interaction between different time instances. The architecture of the neural network is chosen as 4-16-16-1 for both chaotic and non-chaotic cases with $p=1$ and \emph{tanh} as the activation function. The neural network model is trained for 16000 epochs with the Adam optimizer with a learning rate of 0.0001, a mini-batch size of 256 and the last 10\% of training data is split as validation set to monitor generalization.

\subsubsection{A posteriori evaluation of model performance}

For the non-chaotic case, the model performs well for both training and testing data, as shown in \cref{fig:ann_3d_lorenz_nchaos_posteriori}.

\begin{figure}[H]
	\centering
	\includegraphics[scale=0.15]{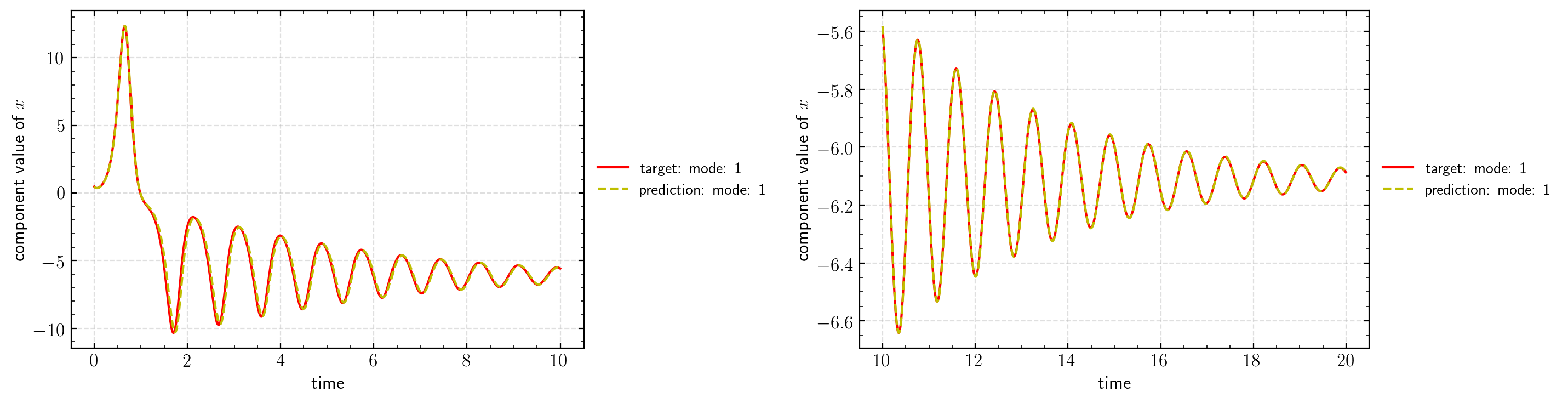}
	\caption{A posteriori model performance on the non-chaotic Lorenz system. Left: training data. Right: testing data.}
	\label{fig:ann_3d_lorenz_nchaos_posteriori}
\end{figure}

For the chaotic case, results are shown in \cref{fig:ann_3d_lorenz_chaos_posteriori} for training and testing evaluations.  \cref{tab:lorenz_lya_cd} shows that both MLE and correlation dimension are in accordance with the truth. Furthermore, Diks criterion shows $|S| = 1.509$ for training data and $|S| = 2.83$ for testing data, which implies that the null hypothesis is accepted. Details of implementation are provided in \cref{app:diks_criterion}. 

\begin{table}[tbhp]
{\footnotesize
\caption{Comparison of MLE and correlation dimension between truth and model}\label{tab:lorenz_lya_cd}
\begin{center}
\begin{tabular}{|c|c|c|} \hline
case          & MLE     & $\gamma$  \\ \hline
true train    & 0.044   & 1.34      \\
model train   & 0.042   & 1.33      \\ \hline
true test     & 0.041   & 1.34      \\
model test    & 0.041   & 1.34      \\ \hline
\end{tabular}
\end{center}
} 
\end{table}

\begin{figure}[H]
	\centering
	\includegraphics[scale=0.15]{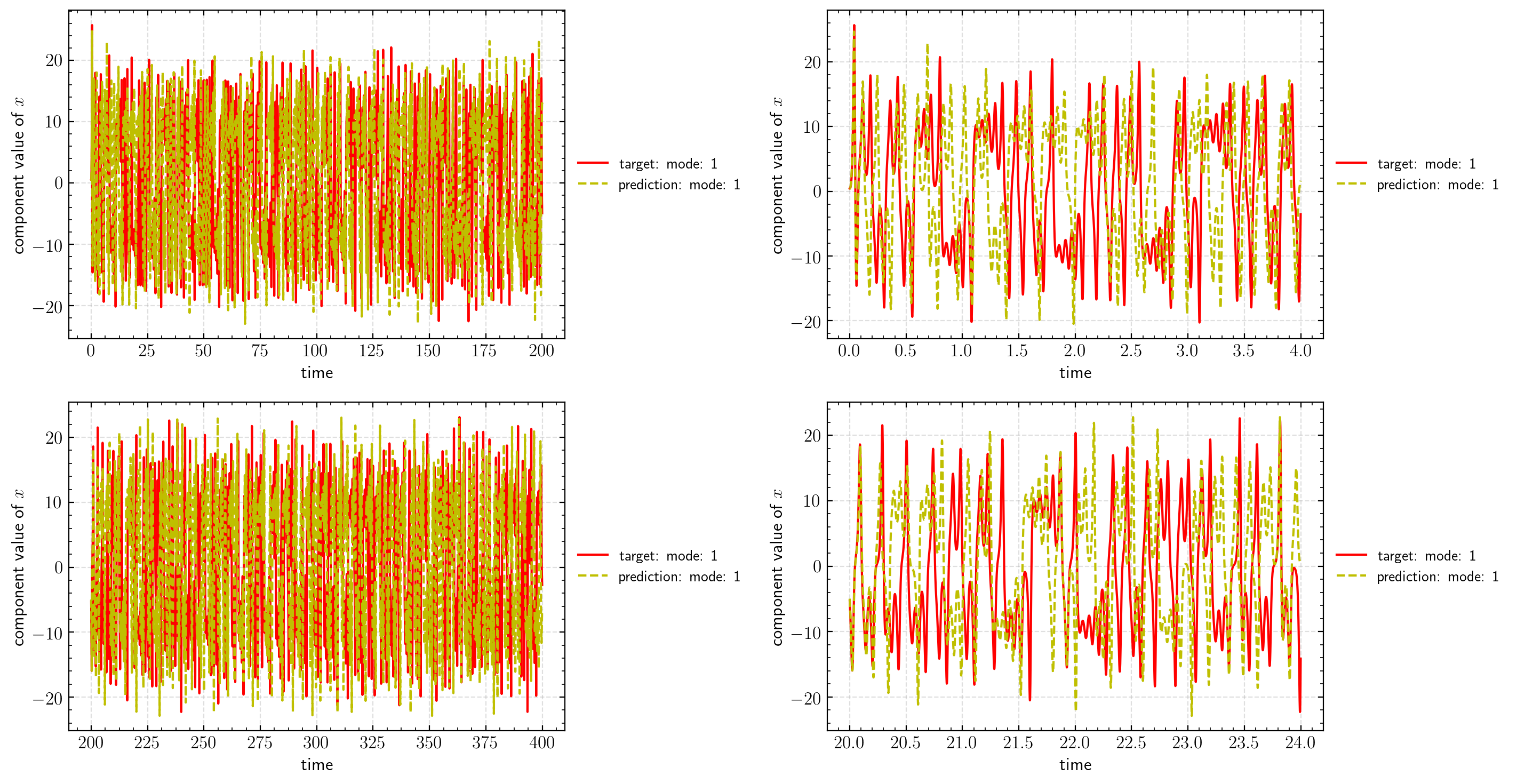}
	\caption{A posteriori model performance on the chaotic Lorenz system. Top left: training data. Bottom left: testing data. Top right: training data zoomed in $t \in [0,4]$. Bottom right: testing data zoomed in $t \in [20, 24]$}
	\label{fig:ann_3d_lorenz_chaos_posteriori}
\end{figure}

\subsection{One dimensional viscous Burgers equation}
\subsubsection{Problem description}
In this section, the one dimensional viscous Burgers equation is considered in a periodic domain  $x \in [0, 2\pi]$ 

\begin{equation}{\label{eq:vbe}}
\frac{\partial  u}{\partial t} + u \frac{\partial u }{\partial x} = \nu \frac{\partial^2 u}{\partial x^2}
\end{equation}
 with $\nu=0.02$, $t\in[0, 10]$, $u(x,0)=\sin(x)$.

Using a standard pseudo-spectral method with two-thirds dealiasing and with Runge-Kutta 3rd order SSP scheme for time stepping, the system is resolved with 1024 grid points uniformly distributed in space, and a time step $\Delta t = 0.01 \Delta x$. 2000 snapshots of $u(x,t)$ are uniformly collected in time. For the setup of coarse graining, we use a spectral filter to obtain the state and corresponding closure with wavenumber $k$ ranging from $-6 $ to $ 5$. The corresponding equation in spectral form for $k^{th}$ wavenumber or mode is 
\begin{align}{\label{eq:vbe_spectral}}
\frac{d \hat{u}_k}{d t} &= -\nu k^2 \hat{u}_k -  \frac{ik}{2}\sum_{p+q=k}\hat{u}_p \hat{u}_q = -\nu k^2 \hat{u}_k - \frac{ik}{2}\sum_{p+q=k,p\in F,q\in F} \hat{u}_p \hat{u}_q \\ \nonumber &- \textcolor{black}{ \frac{ik}{2}\sum_{p+q=k, p\in F, q\in G} \hat{u}_p \hat{u}_q -\frac{ik}{2}\sum_{p+q=k, p\in G, q\in F} \hat{u}_p \hat{u}_q - \frac{ik}{2}\sum_{p+q=k, p\in G, q\in G} \hat{u}_p \hat{u}_q}  ,
\end{align}
where $u(x,0) = \sin(x)$ ; $x \in [0, 2\pi]$; $k \in \{ -N/2, \ldots, N/2-1 \}$; index set of resolved modes $F = \{ -Q/2, \ldots, Q/2-1 \}$; index set of unresolved modes: $G = \{ -N/2, \ldots, -Q/2-1, Q/2 \ldots, N/2-1 \}$. The closure is the sum of last three terms in \cref{eq:vbe_spectral}, and noticing there is a symmetry in the solution with sine wave initial condition, a truncation corresponding to $Q=6$ is 
considered. Only the imaginary part of $\hat{u}_k$ with $k$ ranges from $-6$ to $-1$ is considered. 
For the evaluation of the closure model, we consider $ 0 \le t \le 4 $  as our training data and $ 4 < t \le 10 $ as testing data.

\subsubsection{Model selection}

For the polynomial model, the optimal time delay $p$ and polynomial order $k$ is chosen by sweeping  $p$ from 0 to 2. For each $p$, the optimal $k$ and corresponding $\lambda$ is extracted.

 For the application of the ANN, the best model is selected from a range of hyperparameters with $p$ ranging from 0 to 2. Two hidden layers are used with identical numbers of hidden units for each layer as 4, 8, 12, 16 and type of activation as ReLU, SeLU, tanh. The optimal model was chosen as that which yields the most satisfactory validation result with the smallest number of parameters. We found this to be $p=2$, with 12 hidden units and the tanh activation function. The type of activation function does not appear to be critical in this case, which may be a consequence of the fact that it is not a deep neural network where the vanishing gradient problem may be significant \cite{goodfellow2016deep}.

\subsubsection{A posteriori evaluation of model performance}

As seen in \cref{fig:sindy_6d_vde_posteriori_train,fig:ann_6d_vde_posteriori}, both SINDy and ANN perform well on the training data. When evaluated against unseen testing data, however, performance of SINDy was seen to deteriorate as displayed in \cref{fig:sindy_6d_vbe_posteriori_test}
as a consequence of the extrapolation going out of bounds with high order polynomial features. The ANN is a convergent series of infinite polynomials, and therefore the corresponding model remains less unbounded compared to the polynomial model as shown in \cref{fig:ann_6d_vde_posteriori}.


\begin{figure}[H]
	\centering
	\includegraphics[scale=0.35]{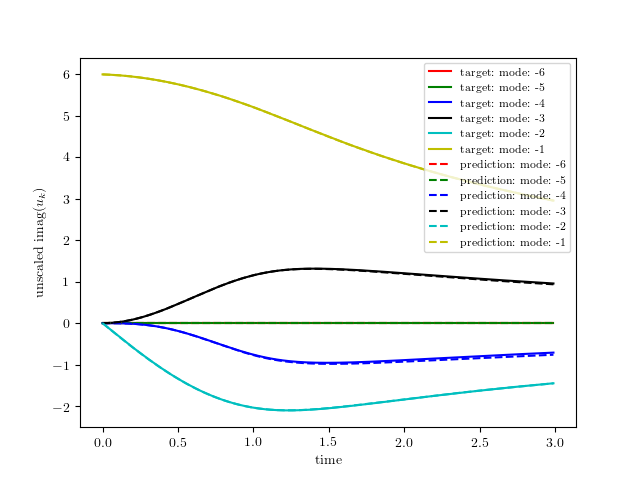}
	\caption{A posteriori model performance on training data : 1D VBE using polynomial closures.}
	\label{fig:sindy_6d_vde_posteriori_train}
\end{figure}
\begin{figure}[H]
	\centering
	\includegraphics[scale=0.35]{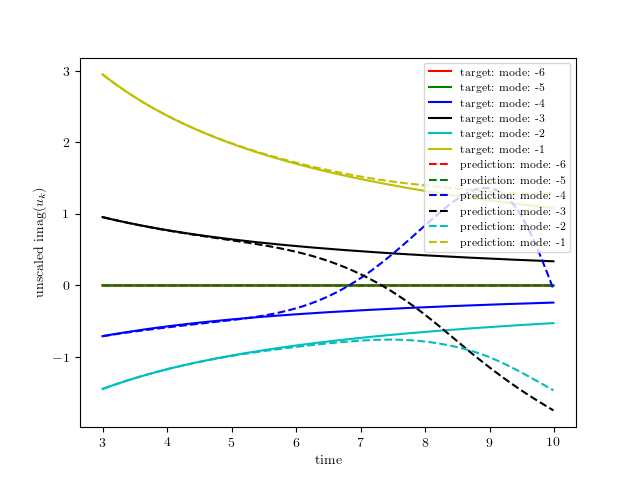}
	\caption{A posteriori model performance on testing data : 1D VBE using polynomial closures.}
	\label{fig:sindy_6d_vbe_posteriori_test}
\end{figure}


\begin{figure}[H]
	\centering
	\includegraphics[scale=0.15]{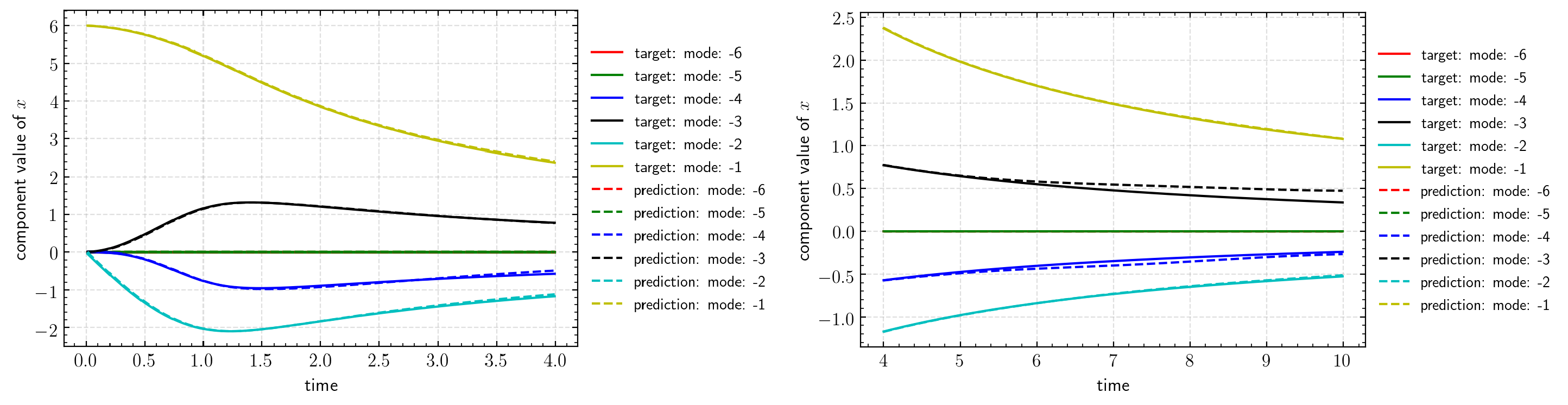}
	\caption{A posteriori model performance on 1D VBE  with ANN. Left: training data. Right: testing data.}
	\label{fig:ann_6d_vde_posteriori}
\end{figure}

Comparison of the results on unseen testing data with ANN at snapshots $t=4.5$, $t=6.0$, $t=7.5$ and $t=9.0$ between the data-driven model, no closure, and ground truth in physical space is shown in \cref{fig:ann_6d_vde_posteriori_physical}. The results highlight the importance of the closure in predicting the future state of this system. The ANN model performs particularly well between $t\in [4, 6]$, with a slight degradation in performance at later times. 

\begin{figure}[H]
	\centering
	\includegraphics[scale=0.15]{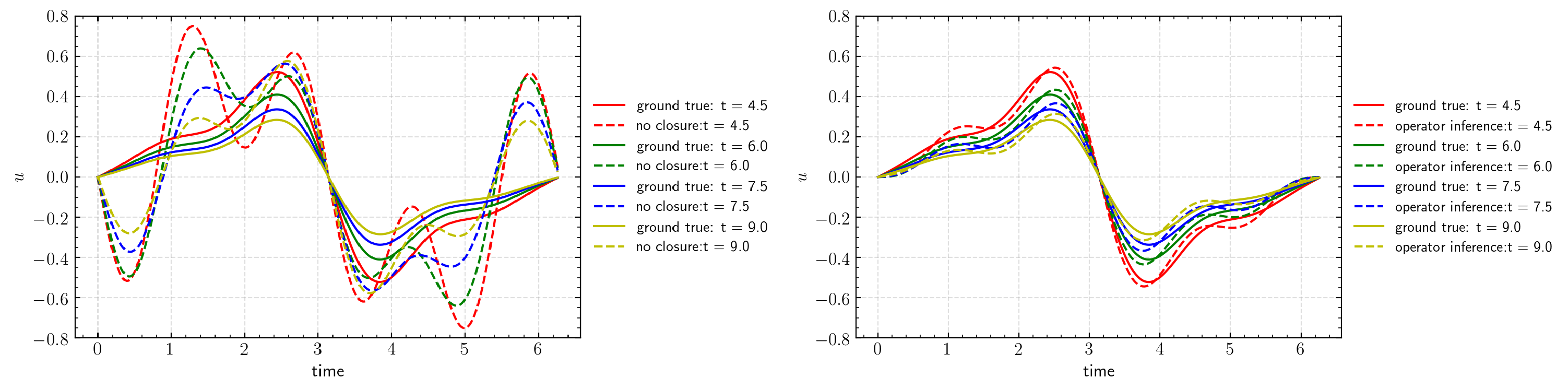}
	\caption{A posteriori model performance on 1D VBE with ANN. Left: without closure. Right: with operator inference closure.}
	\label{fig:ann_6d_vde_posteriori_physical}
\end{figure}

\section{Conclusion}{\label{sec:conclusion}}

An operator inference framework was presented, with the goal of developing closures for reduced models of dynamical systems\footnote{Sample code available at: \url{https://github.com/pswpswpsw/siads_data_driven_closure.git}}. Dynamic memory is embedded into the  equations and the evolution of this term is parametrized via polynomial features and artificial neural networks (ANN). The polynomial model is determined using non-linear regression and \emph{lasso} with Pareto-front-based model selection. The ANN model is determined using gradient-based methods with weight decay regularization. By assuming that different time instances are decoupled from each other, the exponential growth of the number of parameters is limited to a linear growth. For special types of non-linear systems, the closure dynamics was proven to have a compact memory, and the form of the closure is shown to be precisely discoverable using a sparse  set of features. Numerical evaluations of the model on non-chaotic and chaotic dynamical systems are used to evaluate the viability of the procedure, with an emphasis on model selection and a posteriori prediction of unseen data.


\appendix

\section*{Acknowledgements}
This work was supported by AFOSR and AFRL under grants FA9550-16-1-0309 \& FA9550-17-1-0195. The authors thank Mr. Sven Giorno for numerical experimentation, Prof. Cees Diks for discussion on implementing the Diks criterion, and Ms. Helen Zhang for comments on the manuscript. 

\section{Comparison with Elman's recurrent neural network}{\label{app:compare_elman}}

Elman's network\cite{elman1990finding} is one of the earliest~\cite{trebaticky2009prediction} recurrent neural network models,  and was originally proposed to represent temporal structure in linguistics. Although Elman's network is similar to a standard feedforward neural network (FNN), the key difference is that its input includes an additional feedback, and thus the memory effect is addressed in a lossy sense \cite{goodfellow2016deep} using one previous step. 

In this section, we will highlight similarities and differences between the operator inference framework for closure modeling \cref{eq:general_model_1,eq:general_model_2} and Elman's model. Given a general predictive task for a discrete dynamical system: $\{x_i\}_{i=1,\ldots}$, $x_i \in \mathbb{R}^{Q}$, $i \in \mathbb{N}^{+}$, $Q\in\mathbb{N}^{+}$, Elman's network is:
\begin{equation}{\label{eq:elman_1}}
x_{i+1} = \mathcal{C}(h_{i+1}),
\end{equation}
\begin{equation}{\label{eq:elman_2}}
h_{i+1} = \mathcal{H}(h_{i}, x_i),
\end{equation} 
where $\mathcal{H}(\cdot)$ and $\mathcal{C}(\cdot)$ are perceptrons and $h_i \in \mathbb{R}^H$, $H \in \mathbb{N}^{+}$ is the number of units in the context layer. On the other hand, if one considers a simplified discrete case of \cref{eq:general_model_1,eq:general_model_2} with $p=0$, following the same notation, one has:
\begin{equation}{\label{eq:elman_ours_1}}
x_{i+1} = f(x_{i}) + h_{i},
\end{equation}
\begin{equation}{\label{eq:elman_ours_2}}
h_{i+1} = G( h_{i},x_{i}),
\end{equation} 
where $f: \mathbb{R}^Q \mapsto \mathbb{R}^Q$ is known, while $G: \mathbb{R}^Q \times \mathbb{R}^Q \mapsto \mathbb{R}^Q$ is unknown. The similarity is that both \cref{eq:elman_2,eq:elman_ours_2} address the memory effect and extract the dynamics in the same fashion. However, there are at least three different aspects: 
\begin{itemize}
\item Elman's network assumes output dependence only on newly activated hidden units $h_{i+1}$, while our model at $p=0$ considers output dependence on previously activated hidden units $h_{I}$, together with the current input $x_{I}$. Our model also extends the case to $p > 0$, 

\item Our model decouples the evolution processes of hidden units and states while Elman's is formulated in a sequential fashion,

\item  Elman's network requires the determination of all relationships, i.e., the perceptrons, in a purely data-driven fashion, whereas the structure of state evolution is considered known in our operator inference framework. 
\end{itemize}

\section{Implementation of Diks criterion}{\label{app:diks_criterion}}

Diks et al.~\cite{diks1996detecting} developed a test that evaluates whether two attractors are similar. Diks criterion follows statistical inference and can provide probabilistic confidence bounds. In our work, this criterion is used to compare the reconstructed dynamics of an attractor with the ground truth. The method is based on testing a null hypothesis: \emph{two sets of delay vectors are drawn from the same multidimensional probability distribution}. It was later employed by Bakker as a monitoring metric \emph{during} the training  of ANNs for time series modeling. The time series is divided into segments of length $l$ and averaged. To cope with the fractal probability distribution of the chaotic attractor, smoothing is performed via a Gaussian kernel. A bandwidth $d$ is determined by performing sweeps on another trajectory and choosing the $d$ that reveals the highest discrepancy between the two time series. Other hyperparameters are the embedding dimension $m$, and delay time $\tau$. $\tau$ is chosen as the first local minimum of mutual information of Fraser, and $m$ is simply chosen as 2 for the Duffing map and 3 for the Lorenz system. 

Given two sets $\{\bm{X}_i\}_{i=1}^{N_1}$ and $\{\bm{Y}_i\}_{i=1}^{N_2}$ and realizations $\{\bm{x}_i\}_{i=1}^{N_1}$ and $\{\bm{y}_i\}_{i=1}^{N_2}$, the square root of $Q$ defines a distance between the two probability distribution of delay vectors. $\hat{Q}$ is an unbiased estimator of $Q$ and given by:
\begin{equation}{\label{eq:diks_Q}}
\displaystyle \hat{Q} = \frac{1}{\binom
{N_1}{2}} 
{\sum}_{1 \le i < j \le N_1} h(\bm{X}_i,\bm{X}_j) + \frac{1}{\binom{N_2}{2}} \sum_{1 \le i < j \le N_2} h(\bm{Y}_i,\bm{Y}_j) - 
\frac{2}{N_{1} N_2} \sum_{i=1}^{N_1} \sum_{j=1}^{N_2} h(\bm{X}_i,\bm{Y}_j).
\end{equation}

The variance of $\hat{Q}$ under a null hypothesis and conditionally on the set of $N= N_1+N_2$ observed vectors is given by:
\begin{equation}{\label{eq:diks_var_q}}
V_c(\hat{Q}) = \frac{2(N-1)^2(N-2)}{N_1(N_1-1)N_2(N_2-1)(N-3)} \frac{1}{\binom N2}\sum_{1\le i < j \le N}\phi^{2}_{ij},
\end{equation}
where 
$$\phi_{ij} = H_{ij} -g_i -g_j,$$
$$
h(\bm{s}, \bm{t}) = e^{-\lvert \bm{s} - \bm{t} \rvert/4d^2},
$$
and $$H_{ij}=h(\bm{z}_i, \bm{z}_j) - \frac{1}{\binom N2} \sum_{1 \le i < j \le N} h(\bm{z}_i, \bm{z}_j),$$
and $g_i = \frac{1}{N-2}\sum_{j, j\neq i} H_{ij}$,
where $\bm{z}_i$ is defined as 
\[
  \bm{z}_i = 
  \begin{cases}
    \bm{x}_i, & \text{for } 1 \leq i \leq N_1 \\
    \bm{y}_{i-N_1}, & \text{for } N_1 < n \leq N   
    \end{cases}.
\]

Note that $S = \hat{Q}/V_c(\hat{Q})$ is a random variable with zero mean and unit standard derivation under the null hypothesis. As suggested by Diks, we reject the null hypothesis with more than 95\% confidence for $S > 3$. 

In this work, for the Duffing map, the optimal $d=0.0001$, $l$ is chosen as 100, $\tau$ is chosen as 20. For the Lorenz system, the optimal $d=0.001$ and $l$ is chosen as 100, $\tau$ is chosen as 25.

\bibliographystyle{siamplain}
\bibliography{references}

\end{document}